\documentclass[a4paper,12pt,reqno]{amsart}
\usepackage{a4}
\usepackage{ulem}
\usepackage{eucal}
\usepackage{enumerate}
\usepackage[T1]{fontenc}
\usepackage[utf8]{inputenc}
\usepackage{mathtools}
\usepackage{url}
\usepackage{color}
\usepackage[all]{xy}
\usepackage{amsmath,amsfonts,amssymb}
\usepackage{faktor}
\usepackage{hyperref}
\usepackage{cleveref}
\usepackage{ulem}

\usepackage{wasysym}
\usepackage{stmaryrd}

\newcommand\nat{\mathbb N}
\newcommand\zz{\mathbb Z}

\newcommand\mc\mathcal
\newcommand{\wt}{\widetilde}

\swapnumbers
\newtheorem*{thm*}{Theorem}
\newtheorem{thm}{Theorem}[section]
\newtheorem{lem}[thm]{Lemma}
\newtheorem{prop}[thm]{Proposition}
\newtheorem{cor}[thm]{Corollary}

\newtheorem{qu}[thm]{Question}
\theoremstyle{definition}

\newtheorem{ex}[thm]{Example}
\newtheorem{rem}[thm]{Remark}
\numberwithin{equation}{thm}

\newcommand{\car}{\mathsf{char}}

\renewcommand{\min}{\mathsf{min}}
\renewcommand{\max}{\mathsf{max}}

\renewcommand{\leq}{\leqslant}
\renewcommand{\geq}{\geqslant}

\DeclareMathOperator{\Spec}{\mathsf{Spec}}

\DeclareFontFamily{U}{wncy}{}
    \DeclareFontShape{U}{wncy}{m}{n}{<->wncyr10}{}
    \DeclareSymbolFont{mcy}{U}{wncy}{m}{n}
    \DeclareMathSymbol{\Sh}{\mathord}{mcy}{"58} 

\renewcommand{\log}{\mathsf{log}}
\renewcommand{\deg}{\mathsf{deg}}

\renewcommand{\dim}{\mathsf{dim}}

\renewcommand{\cref}{\Cref}
\makeatletter
\DeclareRobustCommand*{\mfaktor}[3][]
{
   { \mathpalette{\mfaktor@impl@}{{#1}{#2}{#3}} }
}
\newcommand*{\mfaktor@impl@}[2]{\mfaktor@impl#1#2}
\newcommand*{\mfaktor@impl}[4]{
   \settoheight{\faktor@zaehlerhoehe}{\ensuremath{#1#2{#3}}}%
   \settoheight{\faktor@nennerhoehe}{\ensuremath{#1#2{#4}}}%
      \raisebox{-0.5\faktor@zaehlerhoehe}{\ensuremath{#1#2{#3}}}%
      \mkern-4mu\diagdown\mkern-5mu%
      \raisebox{0.5\faktor@nennerhoehe}{\ensuremath{#1#2{#4}}}%
}
\makeatother

\author[D.~Grimm, G.~Manzano-Flores]{David Grimm, Gonzalo Manzano-Flores}
\address{Universidad de Santiago de Chile, Facultad de Ciencias, Avenida Libertador Bernardo O'Higgins nº 3363, Estaci\'on Central, Santiago, Chile}
\email{david.grimm@usach.cl}

\address{Universidad de Valpara\'iso, Facultad de Ciencias, Instituto de Matem\'aticas, Gran Bretaña 1111, Valpara\'iso, Chile}
\email{gonzalo.manzano@uv.cl}

\title[Genus inequalities and sums of squares]{Arithmetic genus inequalities  with an application to sums of squares}

\begin{document}
\begin{abstract}\noindent
We show variants of the genus inequality for the irreducible com-
ponents of the special fiber of an arithmetic curve over a henselian discrete
valuation ring of residue characteristic zero that take into account the non-
existence of rational points, respectively real points, on the components. We then
apply this inequality to obtain the bound $ng$ (respectively $n(g+1)$) on the $2$-logarithm of the
totally positive sum-of-two-squares index in the function field of a curve of
genus $g$ with (respectively without) real points, defined over the field of n-fold
iterated real Laurent series. The bound $n(g+1)$ - without the saving in case
of the existence of a real point - has been known in the case of hyperelliptic
curves.
\end{abstract}

\maketitle

\medskip

\noindent
\textit{Classification (MSC 2020):} 11E25, 12D15, 12E30,
12J25, 14H25


\section{Introduction} 

The study of sums of squares in function fields of one variable over a hereditarily pythagorean field $K$ goes back to the 1970's when E.~Becker showed that for any real field $K$, that is, a field $K$ in which $-1$ is not a sum of squares, we have that every sum of squares in $K(X)$ is a sum of two squares if and only if $K$ is hereditarily pythagorean. One implication was generalized in \cite{TY03} to an arbitrary real function field $F/K$ of genus zero, namely if $K$ is hereditarily pythagorean then every sum of squares in $F$ is a sum of two squares (the other implication was shown later in \cite{Gr13}). Naturally, the question arose whether any sum of squares in a function field of a curve of higher genus over a hereditarily pythagorean field is always a sum of $2$ squares. 

Also in the 1970's, L.~Br\"ocker showed in \cite{Br76} that every hereditarily pythagorean field admits a henselian valuation whose residue field has at most two orderings. 

This motivated \cite{TVY04} to study sums of squares of hyperelliptic function fields 
$F$ over $\mathbb{R}(\!(t)\!)$, where they showed that assuming good $t$-adic reduction, every sum of squares in $F$  is a sum of two squares, but they also showed that in the elliptic curve $Y^2=(tX-1)(X^2 +1)$ over $\mathbb{R}(\!(t)\!)$, which is of bad $t$-adic reduction, the function $tX$ is a sum of $3$ squares but not of two squares in its function field. 
In the case of an arbitrary hyperelliptic function field $F$ over a  hereditarily pythagorean field $K$, it was shown in \cite{BVG} that every sum of squares in $F$ is a sum of $4$ squares, and the index of the multiplictive group of nonzero sums of two squares $(F^2+F^2)^\times$ in the multiplicative group of sums of (in this case $4$) squares $\left(\sum F^2\right)^\times$ was studied, as a measure of how far off the sums of $4$ squares are from the sums of $2$ squares.
In the special case where $K=\mathbb{R}(\!(t_1)\!)\ldots (\!(t_n)\!)$, the field of iterated real Laurent series, it is a consequence of \cite[Theorem 3.10]{BVG}  that $$ \left[\left(\sum F^2\right)^\times \, : \, (F^2+F^2)^\times \right]  \leq 2^{n(g+1)},$$ where $F/K$ is the function field of a \textit{hyperelliptic} curve of genus $g$.
In \Cref{mainsostheoremapplied}, we extend this bound to an arbitrary (not necessarily hyperelliptic) function field $F/K$ of genus $g$ as a special case of the more general \Cref{sumsofsquaresbound}. Moreover, we show that $$ \left[\left(\sum F^2\right)^\times \, : \, (F^2+F^2)^\times \right] \leq 2^{ng}$$ whenever $F$ is a real field.  In this way, the bound is naturally optimal in the case of genus zero, by the aforementioned result in \cite{TY03}. In fact, optimality of the bound $2^{n(g+1)}$ in the nonreal case is 
 a consequence of \cite[Example 4.5]{GM24} and in the real case, we show optimality of the bound $2^{ng}$ in \Cref{real optimal example}.  For every $n$ and $g$, optimality is exhibited by the curves $$Y^2=-\prod_{i=0}^{g}(X^{2}+t_{n}^{2i}) \quad \text{ and } \quad Y^2=(X-1)\prod_{i=1}^{g}(X^{2}+t_{n}^{2i}),$$
defined over $\mathbb{R}(\!(t_1)\!)\ldots (\!(t_n)\!)$, respectively.
We should note that both bounds for arbitrary curves of arbitrary genus were already described for $n=1$ in the introduction of \cite{BG24}, as an application of a genus inequality for the reduction of an arithmetic curve shown in that paper. However, this genus inequality is not sufficiently strong to extend the bound on sums of squares for general $n$, as we discuss at the end of \Cref{sumsofsquaressection}. A palatable synthesis of our arithmetic-geometric main results, which allows for an inductive application in the proof for the upper bound of the totally positive sum-of-two-squares index, is the following:

\begin{thm*}[\Cref{nonreal-genus-inequality} \& \Cref{real-genus-inequality}] Let $T$ be a discrete henselian valuation ring with residue field $k$ of characteristic zero and field of fractions $K$. 
Let $F/K$ be a function field in one variable, and assume that $K$ is relatively algebraically closed in $F$. 
Then
$$\sum_{w \in \Omega^{\text{r}}_T(F)} \mathfrak{g}(\kappa_w/k) \,\, + \sum_{w \in \Omega^{\text{n}/\text{n}}_T(F)} \mathfrak{g}(\kappa_w/k) \,\, +  \sum_{w \in \Omega^{\text{n/r}}_T(F)} \!\!1+ \mathfrak{g}(\kappa_w/k) \,\, \leq \,\, \mathfrak{g}(F/K),$$
if $F$ is real, and 
$$\sum_{w \in \Omega^{\text{n}/\text{n}}_T(F)} \mathfrak{g}(\kappa_w/k) \,\,+ \sum_{w \in \Omega^{\text{n/r}}_T(F)} \!\!1+ \mathfrak{g}(\kappa_w/k) \,\, \leq \,\, \mathfrak{g}(F/K) + 1,$$
if $F$ is nonreal.
\end{thm*}
Here, considering the set of all residually transcendental valuation extensions $w$ of $T$, we denote by $\Omega^{\text{r}}_T(F)$ the subset of those $w$ whose residue field $\kappa_w$ is real, and by $\Omega^{\text{n/n}}_T(F)$, respectively $\Omega^{\text{n/r}}_T(F)$,  the subset of those $w$ whose residue field $\kappa_w$ is nonreal, and such that the relative algebraic closure of $k$ in $\kappa_w$ is nonreal, respectively real. By $\mathfrak{g}(F/K)$, respectively $\mathfrak{g}(\kappa_w/k)$, we denote the genus of the function field, which coincides with the genus of the unique regular projective curve $C$ with function field $F$, respectively $\kappa_w$, considered as a curve over the relative algebraic closure of $K$ in $F$, respectively over the relative algebraic closure of $k$ in $\kappa_w$.  In \Cref{rational-genus-inequality} we also show a related version of the genus inequality where we do not require henselianity of $T$, and where the residual genera of valuations  ``without \textit{rational} point in the residual curve'' are the ones that are increased by one. So far, we lack an application for this second version of the genus inequality, but we include it also, since from an arithmetic-geometric point of view, it is the more natural one to consider.

The genus inequality obtained in \cite{BG24}, while additionally admitting valued base fields with perfect residue fields of positive residue characteristic, 
is weaker in the sense that the additional weight "+1" could only be shown for those $w \in \Omega^{n/r}_T(F)$ with $\mathfrak{g}(\kappa_w/k)=0$.
The less restrictive conditions on the valued base field in \cite{BG24} seem to restrict the kind of regular models of an arithmetic curve that one can work with. More precisely, it seems to require to work with the minimal regular model.

In contrast,  in the present paper, we take advantage of the fact that in the motivating application to sums of squares, only (henselian) valued base fields of residue characteristic zero are considered. Under this condition, a cohomological flatness result by Raynaud \cite{Ray70} permits to adapt the methods from \cite{BG24} so that they indeed work with a regular model of an arithmetic curve in which the special fiber is a normal crossing divisor. 

 All previous genus inequalities for constant field reductions in the literature, such as \cite{Mathieu}, \cite{Mat87}, or \cite{GMP89} do not take into account any arithmetic properties of the residue function fields since in these works, $k$ is always assumed to be algebraically closed.
 
Finally, we show in \Cref{boundlocal} that  $$\left[\mathcal{L}(F): F^{\times 2}\right] \leq 2^g,$$
where $\mathcal{L}(F)$ denotes the multiplicative group of local squares in $F^\times$. This is not a consequence of the aforementioned genus inequality, but rather of a complementary topological comparison of the dual graph of the reduction of a certain regular model under maximal unramified base change that we observe in \Cref{bettiinequality}. This,  together with the description of the failure of a local-global principle for quadratic forms in terms of the topology of the special fiber of such a model, given in \cite{HHK15}, yields the bound.
Note that the index of $F^{\times 2}$ inside $\left(F^2 + F^2\right)^\times$ in a function field $F$ is always infinite and thus not interesting, but since $$F^{\times 2} \subseteq \mathcal{L}(F) \subseteq \left(F^2 + F^2\right)^\times,$$ we believe the local squares to be an interesting substitute. Moreover, $\mathcal{L}(F)$ may be identified with the so called Kaplansky radical of $F$, as will be explored in an upcoming article of the authors in collaboration with K.J.~Becher. Both, the bound on the index of the square group inside the local square group, as well as its implications on the Kaplansky radical have been obtained in the doctoral thesis of the second named author at the Universidad de Santiago de Chile \& Universiteit Antwerpen under the joint supervision of the first named author and K.J.~Becher.

\bigskip

 We would like to thank Karim Johannes Becher for his encouragement and discussions on this continuation of \cite{BG24}. We would also like to thank Qing Liu for patiently answering questions of the first named author and making him aware of the crucial cohomological flatness property in residue characteristic zero. 
 The first named author gratefully acknowledges support from \textit{Universidad de Santiago de Chile, proyecto DICYT 042432G}, and the second named author gratefully acknowledges support from \textit{Fondecyt ANID Postdoctoral Grant 3240191} and from Concurso de Subvenci\'on a la Instalaci\'on en la Academia 2025 N°85250089.

\section{Symmetries in graphs}
In the subsequent sections, we will study certain reductions of an algebraic curve, where we aim to bound the number of irreducible components with certain arithmetic properties, such as having a real field of constants but not admitting a real point. These are the relevant irreducible components to bound in view of the genus-inequality theorem for the function field of the algebraic curve, see the stated main theorem in the introduction. Any such irreducible component in the reduction corresponds to what we shall call a singular rigidity in the graph related to the reduction, or more precisely, to a Galois-orbit of singular rigidities, given the natural graph-action of the absolute Galois group of the base field of the reduction. In the current Section, we will define abstractly the notion of a rigidity, and of a singular rigidity as a special case, for a graph with a group action. We give bounds on the number of orbits of rigidities, depending on the Betti-number of the graph. Later on, using an established relation between the genus of an algebraic curve and the genera of the irreducible components of the reduction on one side, and the Betti-number of the reduction graph on the other side, this will almost yield the optimal genus inequalities of the theorem presented in the introduction, except in the case where $F$ is real, the right hand side is off by one. In order to reach the optimal bound in the real case,  we will show that either there exists an additional singular rigidity in the reduction graph which corresponds to an irreducible component admitting real points, or that otherwise there has to exist a certain arrangement of irreducible components with real intersection points that corresponds to a non-singular rigidity in the graph.

\bigskip

A finite (undirected) simple graph $\mathcal{D}$ is a pair $(\mathcal{V}, \mathcal{E})$ consisting of a finite set $\mathcal{V}$ and a subset $\mathcal{E}$ of the set of all subsets of $\mathcal{V}$ of cardinality $2$. An element of $\mathcal{V}$ is called a \textit{vertex} of $\mathcal{D}$ and an element of $\mathcal{E}$ is called a \textit{edge} of $\mathcal{D}$.
For a vertex $v \in \mc{V}$, we call  $\mathrm{deg}(v):= \vert \{w \in \mathcal{V} \mid \{v,w\} \in \mathcal{E}\} \vert$ the \textit{degree of $v$}.
We say that two vertices $v,w \in \mc{V}$ are  \textit{connected} if there is a sequence of vertices $v=v_0, v_1,\ldots, v_n=w \in \mathcal{V}$ such that $\{v_{i-1},v_{i}\} \in \mathcal{E}$ for each $1 \leq i \leq n$.  We say that $\mc{D}$ is \textit{connected}, if any two of its vertices are connected, and \textit{simply connected} if moreover the connecting sequence of vertices between any two given vertices is unique. We set $$\beta(\mc{D}):= 1 -\vert \mc{V} \vert + \vert \mc{E} \vert$$ and call it the \textit{Betti number} of $\mc{D}$.

Let for the rest of this section $\mc{D}$ always denote a finite connected simple graph. Note that $\beta(\mc{D}) \geq 0$ and 
$\beta(\mc{D}) = 0$ if and only if $\mc{D}$ is simply connected. We also call a simply connected graph a \textit{tree}.

We call a finite connected simple graph $\mathcal{D}$ together with a group action of a group $G$  via graph automorphisms a \textit{simple $G$-graph}.
Note that every graph is in particular a $G$-graph for $G$ the group of automorphisms of the graph.
 Let $G$ be a group and $\mathcal{D}=(\mathcal{V},\mathcal{E})$ a simple $G$-graph.  Let $H \leq G$ be a subgroup. We denote by $\mc{D}^H$ the full subgraph of $\mc{D}$ spanned by $\mc{V}^H=\{v \in \mathcal{V} \mid H \subseteq \mathsf{stab}_G(v)\}$.
We call a connected component $\mathcal{R}=(\mathcal{V}_\mathcal{R},\mathcal{E}_\mathcal{R})$ of $\mathcal{D}^H$ an \textit{$H$-rigidity} of the $G$-graph $\mathcal{D}$ if $H=\mathsf{stab}_G(v)$ for every $v \in \mathcal{V}_\mathcal{R}$. A \textit{rigidity of the $G$-graph $\mathcal{D}$} is a $H$-rigidity for some (uniquely determined) subgroup $H \leq G$. 
Note that any two distinct rigidities are disjoint. Given a rigidity $\mathcal{R}$ of $\mathcal{D}$, we sometimes call the common stabilizer subgroup $H$ for all vertices of $\mathcal{R}$ the \textit{rigidifier of $\mathcal{R}$}.

We call a rigidity that consists of exactly one vertex a  \textit{singular rigidity}. They are of particular interest in the current article.

\begin{ex}
Consider the line graph of three vertices $\mathcal{D}=(\mathcal{V}, \mathcal{E})$, where $\mathcal{V}=\{v_1, v_2, v_3\}$ and $\mathcal{E}=\{ \{v_1,v_2\}, \, \{v_2,v_3\}\}$, considered as a $G$-graph, where $G=\mathsf{Aut}(\mathcal{D})$, i.e. the group of order two consisting of the identity and the permutation of $v_1$ and $v_3$. For the trivial subgroup $H=\{\mathsf{id}_\mathcal{D}\}$, we have that $\mathcal{D}^H=\mathcal{D}$ and this graph is connected. However, the full stabilizer of $v_2$ is  all of $G$, so there is no rigidity of $\mathcal{D}$ with rigidifier $\{\mathsf{id}_\mathcal{D}\}$. For the other subgroup $H=G$, we have that that $\mathcal{D}^H=(\{v_2\}, \emptyset)$, which is connected. Moreover the full stabilizer of its only vertex $v_2$ coincides with $H$, so indeed $\mathcal{R}=(\{v_2\}, \emptyset)$ is a (singular) rigidity of $\mathcal{D}$. It is the unique rigidity of the $G$-group $\mathcal{D}$.
\end{ex}

\begin{lem}\label{restricting rigidities}
   Let $G$ be a group and $\mathcal{D}=(\mathcal{V},\mathcal{E})$ a finite simple $G$-graph. Let $\mathcal{D}'=(\mathcal{V}',\mathcal{E}')$ be a connected $G$-stable subgraph of $\mathcal{D}$. Let $\mathcal{R}=(\mathcal{V}_\mathcal{R},\mathcal{E}_\mathcal{R})$ be a rigidity of the $G$-graph $\mathcal{D}$. Then every connected component of $\mathcal{R}':=(\mathcal{V}_{\mathcal{R}} \cap \mathcal{V}' ,  \mathcal{E}_\mathcal{R}\cap \mathcal{E}')$ is a rigidity of the $G$-graph $\mathcal{D}'$.
\end{lem}
\begin{proof}
  Let $H$ be the rigidifier of $\mathcal{R}$. Since the vertices of $\mathcal{V}_{\mathcal{R}}$ have stabilizer $H$, the same is true for the vertices of $\mathcal{V}_{\mathcal{R}} \cap \mathcal{V}'$. Since $\mathcal{V}^H\cap \mathcal{V}'= \mathcal{V}'^H$, we have that every connected component of $\mathcal{R}'$ is a connected component of $\mathcal{D}'^H$, since otherwise, should there exist a connected component of $\mathcal{R}'$ that is properly cotained in a connected component of $\mathcal{D}'^H$, the connected componente of $\mathcal{D}^H$ that contains the one of  $\mathcal{D}'^H$ then would properly contain $\mathcal{R}$, contradicting the hypothesis of being a rigidity of $\mathcal{D}$.
\end{proof}

\smallskip

In \cite{BG24}, the notion of \textit{pivot vertices} was introduced and studied. They were defined as vertices whose stabilizer group acts on the neighboring vertices by separating them into orbits of even cardinality.  In particular, pivot vertices are singular rigidities. The following upper bounds on the number of rigidities or $G$-orbits of rigidities were shown in \cite{BG24} for the special case of pivot vertices. We denote by $\mathcal{D}_G$ the set of rigidities of the $G$-graph $\mc{D}$.
\begin{prop}\label{pivots in trees}
If $\mc{D}$ is a tree, then $|\mc{D}_G| \leq 1$.
\end{prop}
\begin{proof}
If $|\mc{V}|\leq 1$ the statement is trivial.
Assume now that $\mc{D}$ is a tree with $|\mc{V}| > 1$. 
Then $\mc{D}$ has at least one vertex of degree $1$.
If $\mc{D}$ consists only of vertices of degree $1$, that is, if $\mc{D}$ is the tree consisting of two vertices, then 
$\mathcal{D}$ is the unique rigidity of $\mathcal{D}$, independent of the group acting on it.
Otherwise, by removing all vertices of degree $1$ from $\mc{D}$, we obtain a nonempty subtree $\mc{D}'$ of $\mc D$ with strictly fewer vertices than $\mc{D}$. Any rigidity $\mathcal{R}$ of $\mc{D}$ restricts to a unique rigidity of $\mc{D}'$ by \Cref{restricting rigidities}, since any rigidity of $\mathcal{D}$ necessarily contains a vertex that has degree at least $2$ in $\mathcal{D}$. Any two distinct rigidities of $\mc{D}$ are disjoint, and hence restrict to two distinct rigidities of $\mc{D}'$.
Hence the statement follows by induction on $|\mc{V}|$.
\end{proof}

Note that the $G$-action on $\mc{D}$ induces a $G$-action on $\mathcal{D}_G$. More precisely, if $\mc{R} \in \mc{D}_G$ is a rigidity with stabilizer $H$ and rigidifier $\widetilde{H} \leq H$, then $g.\mc{R}$ is a rigidity with stabilizer $gHg^{-1}$ and rigidifier $g\widetilde{H}g^{-1}$,  for any $g\in G$. We denote by $\mathcal{D}_G^G$ the set of rigidities that are fixed under this action of $G$. Note that if $\mc{R} \in \mathcal{D}_G^G$, that is, if $g.\mc{R}=\mc{R}$ for every $g\in G$, then $G$ is not necessarily the rigidifier of $\mc{R}$ in $G$, since the induced $G$-action on the graph $\mc{R}$ need not be trivial. If, however, $\mc{R}$ is a singular rigidity, then indeed $\mc{R} \in \mathcal{D}_G^G$ implies that $G$ is the rigidifier of $\mc{R}$.

\begin{lem}\label{L:G-fixpoint-pivot-count}
Let $\mc{D}^G_G\neq \emptyset$.
Then $|\mc{D}_G | \leq \beta(\mc{D}) + 1$.  
\end{lem}

\begin{proof}
Let  $\mc{R}_0 \in \mc{D}^G_G$. We will show that  $|\mc{D}_G \setminus \{\mc{R}_0\} | \leq \beta(\mc{D})$.
For $v\in\mc{V}$ let $d(v)$ denote the distance between $v$ and $\mc{R}_0$, that is,
the smallest $r\in\nat$ such that there exist $v_0,v_1,\ldots,v_r\in\mc{V}$ with $\{v_{i-1},v_i\}\in\mc{E}$ for $1\leq i\leq r$ where $v_r=v$ and $v_0 \in \mc{R}_0$.
Note that $d(gv)=d(v)$ for all $v\in\mc{V}$ and $g\in G$.

The proof of the statement is  by induction 
 on $|\mc{E}|$.
If $|\mc{E}|=0$, then $\mc{V}= \mc{R}_0=\{v_0\}$ and the statement holds trivially.
Assume now that $|\mc{E}|\geq 1$. We set $m=\max\{d(v)\mid v\in\mc{V}\}$, $\mc{M}=\{v\in\mc{V}\mid d(v)=m\}$ and $\mc{E}^\ast=\mc{E}\cap\{\{w,w'\}\mid w,w'\in\mc{M}\}$.

We first consider the case where $\mc{E}^\ast\neq \emptyset$.
We set $\mc{E}'=\mc{E}\setminus\mc{E}^\ast$ and consider the graph $\mc{D}'=(\mc{V},\mc{E}')$.
Note that $\mc{D}'$ is connected, $\beta(\mc{D}') <\beta(\mc{D})$, and the $G$-action on $\mc{D}$ restricts to a $G$-action on $\mc{D}'$.
Furthermore, every rigidity of $\mc{D}$, after removing edges from $\mc{E}^\ast$, either remains connected and hence remains a single rigidity of $\mc{D}'$, or it decomposes in several connected components, each of which gives a rigidity on $\mc{D}'$, by \Cref{restricting rigidities}.
Since $|\mc{E}'|<|\mc{E}|$, we therefore obtain by the induction hypothesis that 
$|\mc{D}_G\setminus\{R_0\}|\leq |\mc{D}'_G\setminus\{R_0\}|\leq \beta(\mc{D}')< \beta(\mc{D})$.

We now consider the case where $\mc{E}^\ast=\emptyset$. 
Set $\mc{V}'=\mc{V}\setminus \mc{M}$. We consider the graph $\mc{D}'=(\mc{V}',\mc{E}')$ where $\mc{E}'=\mc{E}\cap\{\{v,v'\}\mid v,v'\in\mc{V}'\}$ (i.e.~the full subgraph of $\mc{D}$ spanned by $\mc{V}'$).
Note that $\mc{D}'$ is connected,  $\beta(\mc{D}')\leq \beta(\mc{D})$, the $G$-action on $\mc{D}$ restricts to a $G$-action on $\mc{D}'$, and every rigidity of $\mc{D}$ that contains a vertex in $\mc{V}'$ restricts to at least one rigidity of $\mc{D}'$. 
Since $\mc{M}\neq \emptyset$ we have that $|\mc{E}'|<|\mc{E}|$.
Hence, the induction hypothesis yields that  $|\mc{D}'_G\setminus\{R_0\}|\leq \beta(\mc{D}')$. Let $\mc{M}_G$ denote the set of (necessarily singular) rigidities of $\mc{D}_G$ that contain no vertex in $\mc{V}'$.  Then 
$\vert \mc{D}_G\setminus\{\mc{R}_0\} \vert \leq \vert \mc{D}'_G\setminus\{\mc{R}_0\} \vert + \vert \mc{M}_G \vert \leq \beta(\mc{D}'_G) + \vert \mc{M}_G  \vert$.
Since $\mc{E}^\ast=\emptyset$, we have $|\mc{E}\setminus\mc{E}'| = \sum_{w \in \mc{M}} \deg(w)$ and therefore
$\beta(\mc{D})-\beta(\mc{D}')=\sum_{w\in\mc{M}}(\deg(w)-1)$.
Since every vertex in $\mc{M}_G$ has degree at least $2$, we obtain that 
 $\beta(\mc{D})-\beta(\mc{D}')\geq |\mc{M}_G|$.
Hence we conclude that 
 $|\mc{D}_G\setminus\{R_0\}|\leq |\mc{D}'_G \setminus\{R_0\}|+|\mc{M}_G|\leq \beta(\mc{D})$.
\end{proof}

\begin{lem}\label{L:pivot-betti-bound}
$|\{\mc{R} \in \mc{D}_G \mid \mc{V}_\mc{R} \cap  G.v = \emptyset\}| \,\leq \, \beta(\mc{D})+|G.v|-1$ for any $v \in \mc{V}$.
\end{lem}
\begin{proof}
If $v \in \mc{V}^G$, then in particular $\vert G.v\vert =1$, so the statement of the Lemma coincides with the statement of \Cref{L:G-fixpoint-pivot-count}. Suppose now that $v \notin \mc{V}^G$.
We then apply \Cref{L:G-fixpoint-pivot-count} to a graph obtained from $\mc{D}$ by adding one vertex and connecting it with all vertices in $G.v$:
Let $v_0$ denote an extra vertex, not contained in $\mc{V}$, and set $\mc{V}_0={\mc V}\cup\{v_0\}$ and 
$\mc{E}_0=\mc{E}\cup\{\{v_0,gv\}\mid g\in G\}$.
We obtain that $\mc{D}_0=(\mc{E}_0,\mc{V}_0)$ is a connected graph with $\beta(\mc{D}_0)=\beta(\mc{D})+|G.v|-1$.
By letting $gv_0=v_0$ for all $g\in G$ we extend the $G$-action on $\mc{D}$ to a $G$-action on $\mc{D}_0$.
Note that if $v \notin \mc{V}^G$, then none of the vertices in $G.v$ is part of a rigidity in $\mc{D}_0$, and $\{v_0\}$ is a rigidity of $\mc{D}_0$. Otherwise The connected component of $\mc{D}^G_0$ that contains $v_0$ and $v$ is a rigidity of $\mc{D}_0$, and every vertex $G.v$ is contained in at most one rigidity of $\mc{D}$.
This explains the stated inequality in the case where $v \notin \mc{V}^G$.
\end{proof}

We denote by $\mfaktor{G}{\mc{D}_G}$ the set of $G$-orbits
in $\mc{D}_G$, i.e. the set of $G$-orbits of $G$-rigidities in the $G$-graph $\mc{D}$.

\begin{thm}\label{T:pivot-betti-bound}
Assume that $\mc{D}_G\neq \emptyset$.
 Let $d=\min\{|G/\mathsf{stab}_G(\mc{R})|\mid \mc{R}\in\mc{D}_G\}$.
Then $\vert \mc{D}_G \vert\leq \beta(\mc{D})+2d-1\,.$
Furthermore, $$\left\vert \mfaktor{G}{\mc{D}_G} \right\vert \leq \frac{1}d(\beta(\mc{D})-1)+2.$$
\end{thm}
\begin{proof}
Let $\mc{R}_0 \in \mc{D}_G$ such that its orbit $G.\mathcal{R}_0=\{g.\mc{R}_0 \mid g \in G\}$ in $\mc{D}_G$ has cardinality $d$.

It follows by \Cref{L:pivot-betti-bound} that $$|\mc{D}_G|=|\mc{D}_G \setminus G.\mc{R}_0|+d\leq \beta(\mc{D})+2d-1\,.$$
Then $\vert \mfaktor{G}{\mc{D}_G}\vert \cdot d\leq |\mc{D}_G|\leq \beta(\mc{D})+2d-1$ and
hence $\vert \mfaktor{G}{\mc{D}_G}\vert \leq \frac{1}d(\beta(\mc{D})-1)+2$.
\end{proof}

\begin{cor}\label{border-line vorteces are rational}
 $\left\vert \mfaktor{G}{\mc{D}_G} \right\vert \leq \beta(\mc{D})+1$, and
if equality holds  then $ \mc{D}^G_G={\mc{D}_G}$ or  $\beta(\mc{D})=1$ and $\mc{D}^G_G = \emptyset$.
\end{cor}
\begin{proof}
We may assume $\mathcal{D}_{G} \neq \emptyset $, and we set
$d=\mathsf{min}\{|G/\mathsf{stab}_G(\mathcal{R})|\mid \mathcal{R}\in \mathcal{D}_{G}\}$. Note that $d=1$ if and only if
$\mathcal{D}_{G}^{G} \neq \emptyset$.

So let us first assume that $d\geq 2$. 
By \Cref{T:pivot-betti-bound}, we have that $\vert \mfaktor{G}{\mc{D}_G}\vert  \leq \frac{1}d(\beta(\mc{D})-1)+2$.
If $\beta(\mc{D})\geq 2$  then $\frac{1}{d}(\beta(\mc{D})-1)+2<\beta(\mc{D})+1$, whereby $\vert \mfaktor{G}{\mc{D}_G}\vert \leq \beta(\mc{D})$.
If $\beta(\mc{D})=1$  then $\vert \mfaktor{G}{\mc{D}_G}\vert \leq 2=\beta(\mc{D})+1$.
If $\beta(\mc{D})=0$, then $\mc{D}$ is a tree, and as $\mc{D}_G\neq\emptyset$, it follows by \Cref{pivots in trees} that $|\mc{D}_G|=1=\beta(\mc{D})+1$,  in particular  $\mc{D}_G =\mc{D}_G^G$.

Assume finally that $d= 1$. It follows by \Cref{L:G-fixpoint-pivot-count} that $|\mc{D}_G|\leq \beta(\mc{D}) + 1$, whereby in particular $\vert \mfaktor{G}{\mc{D}_G}\vert \leq \beta(\mc{D})$ if $G$ acts nontrivially on $\mc{D}_G$, so  $\vert \mfaktor{G}{\mc{D}_G}\vert = \beta(\mc{D}) + 1$  implies that $\mc{D}_G=\mc{D}_G^G$. 
\end{proof}

A \textit{bipartite simple graph} is a simple graph that allows a coloring of the vertices by two colors, say cyan and purple, such that every edge is between  vertices of distinct color.

Let us from now on assume that the connected finite simple graph $\mc{D}=(\mathcal{V},\mathcal{E})$ is bipartite. 
Let $\mc{D}'=(\mathcal{V}',\mathcal{E}')$ be another finite simple bipartite graph.
 A \textit{morphism} $\varphi: \mc{D}' \to \mc{D}$ \textit{of bipartite finite simple graphs} $\mathcal{D}$ and $\mathcal{D}'$ is a color-respecting map $\varphi:\mathcal{V}' \to \mathcal{V}$ on the set of vertices such that $\{\varphi(v),\varphi(w)\}\in \mathcal{E}$ whenever $\{v,w \}\in \mathcal{E}'$.
 We say that $\varphi$ is an \textit{epimorphism} if furthermore $\varphi: \mc{V}' \to \mc{V}$ is surjective and for every edge $\{v',w'\} \in \mc{E}'$ there is $\{v,w\} \in \mathcal{E}$ such that $\varphi(v)=v'$ and $\varphi(w)=w'$.
 Write $\mathcal{V}= \mathcal{C} \cup \mathcal{P}$ and $\mathcal{V}'= \mathcal{C}' \dot{\cup} \mathcal{P}'$, where $\mathcal{C}, \mathcal{C}' $ denote the set of cyan vertices and $\mathcal{P},\mathcal{P}'$ the set of purple vertices, respectively.
For $\Gamma \in \mc{C}$ set $e_\Gamma= \vert \varphi^{-1}(\Gamma)\vert$. 
For $x \in \mc{P}$, set $i_x=\vert \varphi^{-1}(x)\vert$ and $$e_x=\frac{1}{\deg(x)} \sum_{\Gamma \in \mc{C} \atop \{\Gamma,x\} \in \mathcal{E}} e_\Gamma.$$

\begin{lem}\label{Betti-epimorphism} Assume that $\mc{D}'$ is connected.
    Let $\varphi: \mc{D}' \to \mc{D}$ be an epimorphism of bipartite finite simple graphs such that every $x\in \mc{P}$ has degree $2$ and moreover $i_x \geq e_\Gamma$ for both $\Gamma \in \mc{C}$ such that $\{x, \Gamma\} \in \mathcal{E}$.
    Then $\beta(\mc{D}) \leq \beta(\mc{D}')$.
\end{lem}
\begin{proof} Note that the connectedness of $\mc{D}'$ implies the connectedness of $\mc{D}$.
If $\beta(\mc{D})=0$, there is nothing to show. Hence, we may assume that $\mc{D}$ is not a tree.
Since for every purple vertex in $\mc{D}$ and $\mc{D}'$ there are two edges, we have that 
$\beta(\mathcal{D})=|\mc{P}|-|\mc{C}|+1$ and
$\beta(\mc{D}')=|\mc{P}'|-|\mc{C}'|+1$. 
In order to show $\beta(\mathcal{D})\leq \beta(\mathcal{D}')$, we need to show that $|\mc{C}'|-|\mc{C}|\leq |\mc{P}'|-|\mc{P}|.$
We first show this in the case where every cyan vertex in $\mathcal{D}$ has degree at least $2$, and we will not require $\mc{D}'$ to be connected in this case. We clearly have that 
$$|\mc{P}'|-|\mc{P}| = \sum_{x\in\mc{P}}(i_x-1)\geq \sum_{x\in\mc{P}}(e_x-1)\geq \sum_{\Gamma\in\mc{C}}(e_\Gamma-1)=|\mc{C}'|-|\mc{C}|.$$
The last inequality may be better understood by considering $$\sum_{x\in\mc{P}}(e_x-1)= \sum_{x\in\mc{P}} \sum_{\Gamma \in \mc C \atop \{\Gamma, x\} \in \mathcal{E}}\frac{1}{2}(e_{\Gamma}-1) \geq \sum_{\Gamma\in\mc{C}}(e_\Gamma-1),$$ where the equality is due to the fact that for every $x \in \mc{P}$ we have $\deg(x)=2$, and hence there are exactly two distinct $\Gamma \in \mc{C}$ such that $\{x, \Gamma\} \in \mc{E}$, and the inequality holds since for every $\Gamma \in \mc{C}$ there exist at least two distinct $x \in \mc{P}$ such that $\{x,\Gamma\} \in \mc{E}$, and since $e_\Gamma -1 \geq 0$ as $\varphi$ is an epimorphism.
We will now prove the statement for the general situation, that is, for a connnected graph $\mc{D}$ that may have cyan vertices of degree one. We will actually not require that $\mc{D}'$ is connected. We will show the statement by induction on the number of vertices of $\mc{D}$. If every cyan vertex in $\mc{D}$ has degree at least $2$, we do not need to involve the induction hypothesis since we have just proven this case. Otherwise, let $\Gamma \in \mc{C}$ be a cyan vertex of degree one, and $x \in \mc{P}$ its unique purple neighbouring vertex. Let $\widetilde{\mathcal{D}}$ denote the full subgraph of $\mathcal{D}$ spanned by all vertices except for $\Gamma$ and $x$, and let $\widetilde{\mathcal{D}'}$ denote the full subgraph of $\mathcal{D}'$ spanned by all vertices except $\varphi^{-1}(\Gamma)$ and $\varphi^{-1}(x)$.
Since both graphs are bipartite, and $\wt{\mc{D}}$ is connected, and the restriction 
$\varphi: \wt{\mc{D}'} \to \wt{\mc{D}}$ is an epimorphism, we have by induction hypothesis that
$$\vert \wt{\mc{C}'}\vert - \vert \wt{\mc{C}} \vert \leq \vert \wt{\mc{P}'}\vert - \vert \wt{\mc{P}}\vert,$$ where $\wt{\mc{C}}, \wt{\mc{C}'}$ and $\wt{\mc{P}}, \wt{\mc{P}'}$ denote the set of cyan and purple vertices in $\wt{\mc{D}}, \wt{\mc{D}'}$, respectively.
On the other hand, we clearly have that
$|\mc{C}'|-|\mc{C}| = \vert \wt{\mc{C}'}\vert - \vert \wt{\mc{C}} \vert + e_\Gamma - 1$
and 
$|\mc{P}'|-|\mc{P}| = \vert \wt{\mc{P}'}\vert - \vert \wt{\mc{P}} \vert + i_x - 1$. The inequality $\vert \wt{\mc{C}'}\vert - \vert \wt{\mc{C}} \vert \leq \vert \mc{P}'\vert - \vert \mc{P}\vert$ now follows from the assumption that $i_x \geq e_\Gamma$.
\end{proof}

\section{Galois-symmetries in the reduction of a regular model with normal crossings}\label{dualgraphsection}

For this section, let $K$ be a field of characteristic zero, $T \subseteq K$ a discrete valuation ring of $K$ and let $k$ denote the residue field of $T$. Let $F/K$ be a function field in one variable, that is, a finitely generated field extension of transcendence degree one. Suppose moreover that $K$ is relatively algebraically closed in $F$.

We call a regular $2$-dimensional integral regular scheme $\mc{X}$ together with a flat projective morphism $\mc{X} \to \Spec(T)$ a \textit{regular model for $F/T$}, if the function field of $\mc{X}$ is $K$-isomorphic to $F$.
We denote by  $\mc{X}_0=\mc{X} \times_T K$ the \textit{generic fiber of $\mc{X}$ over $T$} and by $\mc{X}_s=\mc{X} \times_T k$ the \textit{special fiber of $\mc{X}$ over $T$}. Note that the morphism $\mc{X}_0 \to \mc{X}$ given by projection is an open immersion, and the morphism $\mc{X}_s \to \mc{X}$ is a closed immersion.
By \cite[Lemma~3.2.14]{Liu}, $\mc{X}_0$ is a geometrically integral regular projective curve over $K$ and $\mc{X}_s$ is a geometrically connected projective curve over $k$ by \cite[Corollary 8.3.6]{Liu}.
 We call $\mc{X}$ a \textit{regular model with normal crossings for $F/T$} if the special fiber $\mc{X}_s$, seen as a divisor of $\mc{X}$, is a normal crossings divisor, that is, if it has normal crossing at every closed point $P \in \mc{X}$, that is, if
the irreducible factors of the element in the local ring $\mathcal{O}_{\mc{X},P}$ of $\mc{X}$ at $P$ that define $\mc{X}_s$ locally at $P$ (or more precisely, the pull-back of the divisor $\mc{X}_s$ in $\Spec(\mathcal{O}_{\mc{X},P})$), is part of a minimal set of generators of the maximal ideal $\mathfrak{m}_{\mc{X},P}$ of $\mathcal{O}_{\mc{X},P}$.
Since $\mathrm{char}(K)=0$, a regular model for $F/T$ with normal crossings exists, see \cite[Proposition 10.1.8]{Liu}.

From now on, let $\mathcal{X}$ denote a fixed regular model with normal crossing of $F/T$. We recall that the generic point of every irreducible component $\Gamma$ of $\mc{X}_s$ is of codimension one in $\mc{X}$, whereby its local ring defines a discrete valuation ring of $F$ centered in $T$ at its maximal ideal, and hence it defines a discrete valuation $v_{\Gamma}$ on $F$ extending 
the discrete valuation $v_T: K^\times \to K^\times/\mathcal{O}^\times \simeq \mathbb{Z}$ on $K$ induced by $T$. Moreover, the residue field of each $v_\Gamma$ is $k$-isomorphic to the function field $k(\Gamma)$ of $\Gamma$, and thereby a non-algebraic field extension of $k$.  We denote by $\Omega_T(F)$ the set of all valuation extensions of $v_T$ whose residue field is a non-algebraic extension of $k$. 
Note that by \cite[Proposition 3.1]{BGVG14}, the residue field $\kappa_w$ of any $w \in \Omega_T(F)$ is a function field in one variable over $k$. We denote by $\ell_w$ the relative algebraic closure of $k$ in $\kappa_w$, and we say that $\kappa_w/\ell_w$ \textit{admits a rational place} if there exists a valuation on $\kappa_w$ that is trivial on $\ell_w$ and with residue field $\ell_w$, or equivalently, if the smooth projective curve $C$ over $\ell_w$ with function field $\kappa_w/\ell_w$ admits an $\ell_w$-rational point. 
We set $$\Omega^{\mathrm{rat}}_T(F):=\{w \in \Omega_T(F) \mid \kappa_w/\ell_w \text{ admits a rational place }\}.$$
\begin{rem}\label{algebrorational}
Let $w \in \Omega_T(F)$ be such that $w\neq v_\Gamma$ for every irreducible component $\Gamma$ of $\mc{X}_s$. Then, by \cite[Proposition 3.7]{BGVG14}, we have that
$\kappa_w \simeq \ell(t)$ for some finite extension $\ell/k$ and a transcendental element $t$ over $\ell$. In particular $g(\kappa_w/K)=0$ and $w \in  \Omega^{\text{rat}}_T(F)$.
\end{rem} 

For every closed connected subscheme $C \subset \mc{X}_s$,  the \textit{dual graph of $C$} is defined in \cite[Definition 10.1.48]{Liu} as the graph whose vertices correspond to the irreducible components  of $C$, and between any two distinct irreducible components $\Gamma_1, \Gamma_2$ of $C$, the cardinality of the set of edges between both vertices is defined to be the intersection number $\Gamma_1 \cdot \Gamma_2:=\sum_{P\in \Gamma_1 \cap \Gamma_2} (\Gamma_1 \cdot \Gamma_2)_P$, where $(\Gamma_1 \cdot \Gamma_2)_P= \mathrm{length}(\mathcal{O}_{\mc{X},P}/(p_1,p_2))$, and where $p_1,p_2 \in \mathcal{O}_{\mc{X},P}$ are  irreducible elements that define $\Gamma_1$ and $\Gamma_2$ locally. The positive integer $(\Gamma_1 \cdot \Gamma_2)_P$ is called the local intersection number of $\Gamma_1$ and $\Gamma_2$ at $P$. In general, this definition does not yield a simple graph, even when,  
as is our case, $\mc{X}_s$, and thus also $C$, is a strict normal crossing divisor of $\mc{X}$. Nevertheless,
 the local intersection number at a point $P \in \mc{X}_s$ where any two irreducible components of $\mc{X}_s$  intersect, is one for those two components,  and zero for any other pair of irreducible components of $C$, see \cite[Proposition 9.1.8]{Liu}. In particular, the local  parameters of the irreducible components at $P$ generate the maximal ideal of the local ring at $P$, and at most two irreducible components of $C$ intersect at $P$, so in our case, this graph is a simple graph unless two irreducible components of $C$ intersect in more than one point.  Moreover, every edge between two vertices corresponds to a unique intersection point between both corresponding irreducible components of $C$. 
 
 We now adapt the notion of the dual graph slightly, by assigning a colour, say cyan, to the vertices corresponding to irreducible components  and introducing an additional set of vertices of a different color, say purple, where a purple vertex corresponds to an intersection point between two irreducible components of $C$, and the edges of the adapted dual graph are pairs consisting of an irreducible component $\Gamma$ of $C$ together with an intersection point of $\Gamma$ with a distinct irreducible component of $C$. This way we define indeed a  simple bipartite  graph $\mc{D}(C)$. We call this graph the \textit{bipartite dual graph of $C$}. Its Betti number clearly coincides with that of the dual graph as defined after \Cref{algebrorational}.
 Note that all purple vertices in the bipartite dual graph of $C$ have degree $2$.
 
\begin{rem}
We defined the Betti number at the beginning of the second section for finite simple graphs. However, the same formula is used more generally in the literature to define the Betti number for any finite graph, such as for example the dual graph of $C$. We also observe that the Betti number of the dual graph of $C$ coincides with the Betti number of its bipartite dual graph, as defined above.
\end{rem}
Let $T'$ be a maximal unramified valuation ring extension of $T$ inside an algebraic closure of $K$. Denote by $K'$ and $k'$  field of fraction and residue field of $T'$ respectively. We denote 
$\mc{X}':= \mc{X} \times_T T'$. 

\begin{lem}\label{bettiinequality}
 $\mc{X}'$ is a regular model with normal crossings for $FK' /T'$. Moreover $\beta(\mc{D}(\mc{X}_s))\leq \beta(\mc{D}(\mc{X}'_s))$.
\end{lem}

\begin{proof} Since $\mathcal{X}$ is flat projective over $T$, we have that $\mathcal{X}'$ is  flat projective  over $T'$, by \cite[Corollary 3.3.32]{Liu} and \cite[Proposition 4.3.3, (e)]{Liu}. As $\mathcal{X}'_0\simeq_{K'}\mathcal{X}_{0}\times_{K}K'$ and $\mathcal{X}'_s \simeq_{k'} \mathcal{X}_s \times_k k'$, we have that both are curves over $K'$  and $k'$, respectively, by \cite[Proposition 3.2.7]{Liu}. Hence $\dim(\mathcal{X}')=2.$ We claim furthermore that $\mathcal{X}'$ is regular. By \cite[Corollary 4.2.17]{Liu} it is enough to show regularity only for closed points. Let $x\in\mathcal{X}_s \hookrightarrow \mc{X}$ be a closed point. Since $T'$ is an unramified extension of $T,$ the  projection $\pi:\mathcal{X}'\to\mathcal{X}$ is unramified by \cite[Proposition 4.3.22]{Liu}. Hence $\mathfrak{m}_{\mathcal{X}',x}=\mathfrak{m}_{\mathcal{X},\pi(x)}\mathcal{O}_{\mathcal{X}',x},$ whereby the regularity of $\mathcal{O}_{\mathcal{X},\pi(x)},$ implies the regularity of $\mathcal{O}_{\mathcal{X}',x}.$ 
To show that furthermore $\mc{X}'$ has normal crossings at $x$,  let $t\in T$ denote a uniformizer of $T$. Since the closed subscheme $\mathcal{X}_{s}$ of $\mc{X}$ can be defined by the sheaf of ideals in $\mathcal{O}_{\mc{X}}$ generated by $t$, we have by the normal crossing property of $\mc{X}$ at $p(x)$ that $t=p_{1}^{n_{1}}p_{2}^{n_{2}}$ for some $n_{1},n_{2}\in\mathbb{N}$ and some irreducible elements $p_{1},p_{2}\in\mathcal{O}_{\mathcal{X},\pi(x)}$ that generate the maximal ideal $\mathfrak{m}_{\mathcal{X},\pi(x)}.$ Since $t$ is also a uniformizer of $T'$, we have that $\mathcal{X}'_{s}$ is the closed subscheme of $\mathcal{X}'$ defined by the sheaf of ideal generated by $t$ in $\mathcal{O}_{\mathcal{X}'}.$ Considering the equality $t=p_{1}^{n_{1}}p_{2}^{n_{2}}$ in $\mathcal{O}_{\mathcal{X}',x},$ and the fact that $p_{1}$ and $p_{2}$ also generate the maximal ideal $\mathfrak{m}_{\mathcal{X}',x}$ in $\mathcal{O}_{\mathcal{X}',x}$,  we conclude that $\mathcal{X}'_{s}$ has normal crossing at $x.$ 

The base change $\pi_s: \mc{X}'_s \to \mc{X}_s$ of $\pi$ induces a epimorphism  $\varphi: \mc{D}(\mc{X}'_s) \to \mc{D}(\mc{X}_s)$ of bipartite connected graphs. Recall that since both special fibers have normal crossing in  $\mc{X}$ and $\mc{X}'$, respectively, we have that every purple vertex in either graph has degree two.
Moreover, if $\Gamma$ is an irreducible component of $\mc{X}_s$ and $x \in \Gamma$ a point of intersection with another irreducible component, then we have that $\ell_\Gamma \subseteq \kappa(x)$, and thus in view of \Cref{Betti-epimorphism} that $i_x=[\kappa(x): k] \geq [\ell_\Gamma :k]= e_\Gamma$, whereby  $\beta(\mc{D}(\mc{X}_s))\leq \beta(\mc{D}(\mc{X}'_s))$.
\end{proof}

Let $G$ denote the absolute Galois group of $k$. It acts naturally by $k$-scheme automorphisms on $\mc{X}'_s$ via the canonical identification $\mc{X}'_s\simeq \mc{X}_s \times_k k'$. In particular, $G$  acts on the set of irreducible components and on the set of intersection points of distinct irreducible components. This also induces naturally a Galois action on the bipartite dual graph $\mc{D}(\mc{X}'_s)$.

\smallskip

We will now apply the general graph-theoretic symmetry considerations from the previous section to the bipartite dual graph $\mathcal{D}(\mc{X}'_s)$.
We endow any closed connected subset $Y \subseteq \mc{X}_s$ with its induced reduced subscheme structure $(Y,\mathcal{O}_Y)$, and we set $\ell_Y:=H^0(Y, \mc{O}_Y)$, i.e. the ring of global regular functions on $Y$. 
Since $Y$ is connected and reduced, we have that $\ell_Y$ is a finite field extension of $k$, see \cite[Corollary 3.3.21]{Liu}. We denote by $\mc{E}_Y$ the set of $k$-embeddings $\sigma: \ell_Y \to k'$.

\begin{lem}\label{decomposition base change} For any closed connected subset $Y$ of $\mc{X}_s$,  we have that  $$Y \times_k k' \simeq \bigsqcup_{\sigma \in \mc{E}_Y} Y \times_{(\ell_Y, \sigma)} k',$$
is a disjoint union of distinct closed connected subschemes $Y'_\sigma = Y \times_{(\ell_Y, \sigma)} k'$ of $\mc{X}'_s$. The natural $G$-action on $\mc{X}'_s$ induces a transitive $G$-action on $\{ Y'_\sigma \mid  \sigma \in \mc{E}_Y \}$ corresponding to the natural $G$-action on $\mc{E}_Y$ with stabilizers $\mathrm{stab}_{Y'_\sigma} = \mathrm{Gal}(k'/\sigma(\ell_Y))$ for $\sigma \in \mc{E}_Y$.
Moreover, if $Y$ is irreducible, then so is $Y'_\sigma$ for every $\sigma \in \mc{E}_Y$. 
\end{lem}
\begin{proof}
The reduced closed subscheme $Y$ has a natural structure over $\ell_Y$, and by  \cite[\href{https://stacks.math.columbia.edu/tag/0FD1}{Lemma 0FD1}]{stacks-project}, we  have that $Y\times_{(\ell_Y,\sigma)} k'$ is connected for any $\sigma \in \mathcal{E}_Y$. If $Y$ is irreducible, then 
$\ell_Y$ coincides with the relative algebraic closure of $k$ in the function field of $Y$, since $Y$ is smooth by the normal crossing hypothesis on $\mc{X}_s$ and thus $(Y, \mathcal{O}_Y)$ is a normal curve.

Let $U=\mathrm{Spec}(A)$ be an arbitrary open dense affine subscheme of $Y$. Then $A$ is a finitely generated $k$-algebra, and thus in particular finitely generated as an $\ell_Y$-algebra.  In case $Y$ (and hence $U$) is irreducible, we have that $A$ is an integral domain and $\ell_Y$ is relatively algebraically closed in $A$, whereby $A \otimes_{(\ell_Y,\sigma)} k'$ is a domain, thus $Y \times_{(\ell_Y,\sigma)} k' $ is irreducible, for any $\sigma \in \mathcal{E}_Y$.

In order to study the Galois action on $Y\times_k k'$, it is enough to study it on $U \times_k k'$.
The latter given by the spectrum of $$A \otimes_k k'= (A \otimes_{\ell_Y} \ell_Y) \otimes_k k' =A \otimes_{\ell_Y} (\ell_Y\otimes_k k')= A \otimes_{\ell_Y} \prod_{i=1}^n k',$$
where 
$n=[\ell_Y : k]$ and  $\prod_{i=1}^n k'$ is considered as an $\ell_Y$-algebra via the diagonal embedding $x \mapsto (\sigma_1(x), \ldots, \sigma_n(x))$ of $\ell_Y$ with  $\mc{E}_Y=\{\sigma_1,\ldots, \sigma_n\}$. 
Hence, $A \otimes_k k' \simeq \prod_{i=1}^n A \otimes_{(\ell_Y, \sigma_i)} k'$. The natural action of $G$ on $A\otimes_k k'$ where  $\sigma \in G$ acts via $\mathrm{id}_A \otimes \sigma$, induces  an action on  $\prod_{i=1}^n A \otimes_{(\ell_Y, \sigma_i)} k'$  where $\sigma \in G$ acts on 
$(x_1,\ldots, x_n)$ by $\sigma. (x_1,\ldots, x_n)= (\sigma(x_{\sigma(1)}), \ldots, \sigma(x_{\sigma(n)}))$, where $\sigma(i)$ denotes the index $j$ such that
$\sigma_j = \sigma \circ \sigma_i$. In particular $G$ induces a transitive permutation action on the set of Cartesian factors.
 The stabilizer subgroup in $G$ of the Cartesian factor $A \otimes_{(\ell_Y, \sigma_i)} k'$  is  $\mathrm{Gal}(k'/\sigma_i(\ell_Y))$. Since $U \subseteq Y$ was an arbitrary open affine subscheme of $Y$, we conclude that  $Y\times_{(\ell_Y, \sigma_i)} k'$ and $Y\times_{(\ell_Y, \sigma_j)} k'$ are disjoint in $Y\times_k k'$  and thus in $\mc{X}'_s$. 
 \end{proof}

\begin{prop}\label{rigidity-class} 
The set of $G$-orbits of cyan $G$-rigidities of $\mc{D}(\mc{X}'_s)$ is in bijection with the set of reduced connected closed sub-curves  $Y$ of $\mc{X}_s$ which satisfy the following properties: \\
There exists a finite normal field extension $\ell/ \ell_Y$ such that
\begin{enumerate}
\item[$(i)$]  $\ell \simeq_k \ell_\Gamma$ for every irreducible component $\Gamma$ of $Y$ and,
\item[$(ii)$] The complement in $Y$ of all intersection points $P$ between distinct irreducible components such that $\ell \subsetneq k(P)$, is connected.
\item[$(iii)$] If $\Delta$ is an irreducible component of $\mc{X}_s$ such that $\Delta(\ell) \cap Y(\ell) \neq \emptyset$, then $\Delta \subset Y$.
\end{enumerate}
Under this bijection, the $G$-rigidities of a $G$-orbit are singular if and only if $Y$ is irreducible.
\end{prop}

\begin{proof}
Let $Y$ be a closed reduced connected subcurve of  of $\mc{X}_s$ satisfying the properties and let $Y'=Y\times_k k'$.
By \Cref{decomposition base change}, we have  $$ Y'= \bigsqcup_{\tau \in \mc{E}_Y} Y \times_{(\ell_Y, \tau)} k',$$
is a decomposition in connected components where $\mc{E}_Y$ denote the distinct $k$-embeddings of $\ell_Y$ into $k'$. Moreover, the natural $G$-action on $Y'=Y\times_k k'$ is the restriction of the natural $G$-action of $\mc{X}'_s$ and its induced action on the set of connected components $Y \times_{(\ell_Y, \tau)} k'$ corresponds to the natural  $G$-action on the index-set $\mc{E}_Y$, which is transitive. 
We are thus left to show that for each $k$-embedding $\tau: \ell_Y \to k'$, the connected component
$Y'_\tau:=Y \times_{(\ell_Y, \tau)} k'$ of $Y\times_k k'$ corresponds to a $G$-rigidity of $\mc{D}(\mc{X}'_s)$  with stabilizer $\mathrm{Gal}(k'/ \tau(\ell_Y))$ and rigidifier $\mathrm{Gal}(k'/ \tau(\ell))$ for the unique homomorphic continuation  $\tau: \ell \to k'$  of $\tau: \ell_Y \to k'$.  Property $(iii)$ guaranties that  $Y\times_{(\ell_Y, \tau)} k'$ defines a full connected component in the $\mathrm{Gal}(k'/\tau(\ell))$-invariant subgraph of $\mc{D}(\mc{X}'_s)$, and $(i)$ and $(ii)$ furthermore characterize that it is indeed a $G$-rigidity with rigidifier $\mathrm{Gal}(k'/\tau(\ell))$, as is clear from \Cref{decomposition base change}.

Now assume conversely that 
$Y'_1, \ldots, Y'_d$ is the collection of connected curves of $\mc{X}'_s$ that corresponds to a $G$-\'orbit of $G$-rigidities.
Then, for the corresponding connected closed subcurve $Y$ in $\mc{X}_s$, we can identify $\ell_Y$ with the fixed subfield in $k'$ of the stabilizer subgroup $\mathrm{stab}_G(Y'_1)$ of $G$. Let $H \subseteq \mathrm{stab}_G(Y'_1)$ be the rigidifier of the rigidity that corresponds to $Y'_1$ and let $\ell$ denote its fixed field in $k'$. We are left to show that $\ell/ \ell_Y$ is normal. 

For the closed immersion $\iota_\Gamma: \Gamma \hookrightarrow Y$ of an irreducible component of $Y$ let $ \iota^*_{\Gamma}: \ell_Y \hookrightarrow \ell_\Gamma$ be the corresponding $k$-embedding, and for the closed immersion of a closed point $\iota_{\Gamma,P}: \{P\} \hookrightarrow \Gamma$ let  $\iota_{\Gamma,P}^*: \ell_\Gamma \to k(P)$ denote the corresponding $k$-embedding. These homomorphisms fit together in such a way that if $P$ is a point of intersection for two irreducible components $\Gamma_1$ and $\Gamma_2$ of $Y$, then $\iota_{\Gamma_1,P} \circ \iota_{\Gamma_1} =\iota_{\Gamma_2,P} \circ \iota_{\Gamma_2}$.

Considering that $Y$ is connected, the irreducible components of $Y$ together with the closed point of intersections of distinct irreducible components yield a large connected commutative diagram where $\ell_Y$ can be identified with the pushout of this diagram, that is, an element of $x\in \ell_Y$ corresponds to a tuple $$(x_\Gamma)_{_{{\Gamma  \subseteq Y} \atop \text{irr.} }} \in  \prod_{\Gamma \subseteq Y \atop \text{irr.}} \ell_\Gamma$$
with the matching-property that for any two irreducible components  
$\Gamma_1, \Gamma_2 \subseteq Y$ with nonempty intersection and every  $P_{\{1,2\}} \in  \Gamma_1\cap \Gamma_2$ we have $\iota^*_{\Gamma_1,P_{\{1,2\}}}(x_{\Gamma_1}) = \iota^*_{\Gamma_2,P_{\{1,2\}}} (x_{\Gamma_2})$. Let us justify this description of elements in $\ell_Y$:
We have that $\Gamma \times_k k' = \bigsqcup_{(\sigma \in \mc{E}_\Gamma)} \Gamma \times_{(\ell_\Gamma, \sigma)} k'$, and $\Gamma \times_{(\ell_\Gamma, \sigma)} k'$ is an irreducible component of $Y \times_{(\ell_Y, \tau)} k'$  for every $k$-embedding $\sigma: \ell_\Gamma \to k'$ such that $\sigma \circ \iota^*_{\Gamma}  = \tau$. 
The stabilizer subgroup of one of the connected components of $Y'$, say $Y \times_{(\ell_Y, \tau)} k'$, is clearly $\mathrm{Gal}(k'/ \tau(\ell_Y))$. Let us identify $\ell_Y$  a priori inside of $k'$ and $\tau$ as the identity inclusion.
Let $x \in \ell_Y=H^0(Y, \mathcal{O}_Y)$. Clearly, for every irreducible component $\Gamma$ of $Y$, the closed immersion $\iota_{\Gamma}: \Gamma \hookrightarrow Y$ yields a homomorphism $\iota_{\Gamma}^*: \ell_Y=H^0(Y, \mathcal{O}_Y) \to H^0(\Gamma, \mathcal{O}_\Gamma)=\ell_\Gamma$, and so by setting $x_\Gamma:= \iota^*_{\Gamma}(x)$, we obtain in this way a tuple in $\prod_{\Gamma \subseteq Y \atop \text{irr.}} \ell_\Gamma$ with the aforementioned matching-property. 
 Conversely, every tuple $(x_\Gamma)_\Gamma \in \prod_{\Gamma \subseteq Y \atop \text{irr.}} \ell_\Gamma$ satisfying the  aforementioned matching-property defines an element  $x \in \ell_Y=H^0(Y, \mathcal{O}_Y)$. 
 
 Note that by assumption $k(P_{\{1,2\}})\simeq_k \ell \simeq_k H^0(\Gamma_i, \mathcal{O}_{\Gamma_i})$ for $i=1,2$,  whereby $\iota_{\Gamma_i,P_{\{1,2\}}}$ are isomorphisms.
Fixing one irreducible component $\Gamma_0 \in Y$, and identifying $\ell=H^0(\Gamma_0, \mathcal{O}_{\Gamma_0})$, we can thus identify $H^0(Y, \mathcal{O}_Y)$ as the set of elements $x \in \ell$ such that for any closed path $\gamma= (\Gamma_0, P_{\{0,1\}}, \Gamma_1, P_{\{1,2\}}, \ldots, P_{\{n-1,n\}}, \Gamma_n, P_{\{n,0\}}, \Gamma_0)$ in the connected subgraph of $\mc{D}(\mc{X}_s)$ corresponding to $Y$, we have  $\sigma_\gamma(x)=x$ where $$\sigma_\gamma=\left(\iota_{\Gamma_0, P_{\{0,1\}}} \circ \iota^{-1}_{\Gamma_1, P_{\{0,1\}}} \right) \circ  \ldots, \circ \left(\iota_{\Gamma_n, P_{\{0,n\}}} \circ \iota^{-1}_{\Gamma_0, P_{\{0,n\}}} \right)$$
is an $k$-automorphism of $\ell$. Hence, $\ell_Y \subseteq \ell$ is the subfield of invariants of a set of automorphisms of $\ell$, and hence $\ell/\ell_Y$ is a Galois extension. 
\end{proof}

Let  $N(\mc{D}(\mc{X}'_s))$ denote the number of $G$-orbits of cyan singular $G$-rigidities $\mc{D}(\mc{X}'_s)$.
Denote 
$$\Omega(\mc{X}_s):=\{\Gamma \subseteq \mc{X}_s \mid \Gamma \text{ irreducible component}\},$$
and $$\Omega^{\mathrm{rat}}_{\mathrm{int}}(\mc{X}_s):=\{\Gamma \in \Omega(\mc{X}_s) \mid \exists \widetilde{\Gamma}\in \Omega(\mc{X}_s)\setminus\{\Gamma\}: \;  \Gamma(\ell_\Gamma) \cap \widetilde{\Gamma}(\ell_\Gamma)\neq \emptyset \},$$
i.e. the set of irreducible components of $\mc{X}_s$ that intersect some other irreducible component of $\mc{X}_s$ in a point that is rational for the first mentioned component. With this notation, we obtain directly from the second part of \Cref{rigidity-class} the following: 
 
\begin{cor}\label{cyan vorteces are nonrats}
    $N(\mc{D}(\mc{X}'_s))=\vert \Omega(\mc{X}_s) \setminus \Omega^{\mathrm{rat}}_{\mathrm{int}}(\mc{X}_s)\vert$.
\end{cor}

We say that a property $\mathcal{P}$ of fields  \textit{satisfies going up} if for any place $M \dashrightarrow L$ (i.e. a homomorphism $M \hookrightarrow L$ or the residue homomorphism $\mathcal{O}_M \to L$ of a valuation ring $\mathcal{O}_M$ of $M$ with residue field $L$), we have that $\mathcal{P}(M)$ implies $\mathcal{P}(L)$.
We say that the property $\mathcal{P}$ \textit{satisfies going down} if $\mathcal{P}(L)$ implies $\mathcal{P}(M)$. 

\begin{ex} The field property $\mathcal{P}=\text{``is nonreal''}$, i.e. ``$-1$ is a sum of squares in the field'', satisfies going up.
The empty property, say $\mathcal{P}=\text{``$1\neq 0$''}$, also satisfies going up.
\end{ex}

Let $\mathcal{P}$ be a property of fields that satisfies going up. Then its negation $\neg \mathcal{P}$ satisfies going down.
We say that a finite field extension $\ell/k$  is \textit{$\neg \mathcal{P}$-minimal in $\mc{X}_s$} if $\neg\mathcal{P}(\ell)$ is true and $\mc{X}_s(\ell)\neq \emptyset$, and if furthermore $\mc{X}_s(\wt{\ell})= \emptyset$ for every proper subfield $\wt{\ell} \subsetneq \ell$ containing $k$.
We denote 
$$\Omega^{\mathcal{P}}_T(F):=\{ w \in \Omega_T(F) \mid \mathcal{P}(\kappa_w)  \text{ is true}\}$$
and
$$\Omega^{\mathcal{P}}(\mc{X}_s):=\{ \Gamma \in \Omega(\mc{X}_s) \mid  \forall \ell/k \, :  \,  \Gamma(\ell)\neq \emptyset \Rightarrow \mathcal{P}(\ell)\}.$$

\begin{rem}  If $\Gamma \in \Omega(\mc{X}_s)$ is such that $v_\Gamma \in \Omega^{\mathcal{P}}_T(F)$, then $\Gamma \in \Omega^{\mathcal{P}}(\mc{X}_s)$.
\end{rem}

\begin{lem}\label{non-rat vorteces}
Let $\mathcal{P}$ denote a field property that satisfies going up. Then
$$ \vert \Omega^{\mathcal{P}}_T(F)  \setminus \Omega^{\mathrm{rat}}_T(F)  \vert  \, \leq  \,\vert \Omega^{\mathcal{P}}(\mc{X}_s) \setminus \Omega^\mathrm{rat}_{\mathrm{int}}(\mc{X}_s) \vert \, \leq \, \beta(\mathcal{D}(\mc{X}'_s)) + 1.$$ 

\smallskip

Moreover, these inequalities are strict in the following cases:
 \begin{enumerate}
     \item[$(i)$] If there exists $\Gamma \in \Omega^{ \mathcal{P}}(\mc{X}_s) \setminus \Omega^\mathrm{rat}_{\mathrm{int}}(\mc{X}_s)$ such that $\Gamma(\ell_{\Gamma}) \neq \emptyset$ then 
     $$ \vert \Omega^{\mathcal{P}}_T(F)  \setminus \Omega^{\mathrm{rat}}_T(F)  \vert  \, <  \,\vert \Omega^{ \mathcal{P}}(\mc{X}_s) \setminus \Omega^\mathrm{rat}_{\mathrm{int}}(\mc{X}_s) \vert.$$
     \item[$(ii)$] If there exist $\Gamma \neq \widetilde{\Gamma} \in \Omega^{\mathrm{rat}}_{\mathrm{int}}(\mc{X}_s)$ with
     $\Gamma(\ell) \cap \widetilde{\Gamma}(\ell) \neq \emptyset$  for a finite field extension $\ell/k$ that is $\neg\mathcal{P}$-minimal in $\mc{X}_s$, then 
    $$\vert \Omega^{\mathcal{P}}(\mc{X}_s) \setminus \Omega^\mathrm{rat}_{\mathrm{int}}(\mc{X}_s) \vert \, < \, \beta(\mathcal{D}(\mc{X}'_s)) + 1.$$
 \end{enumerate}
\end{lem}
\begin{proof}
    The first inequality follows from \Cref{algebrorational}, and obviously the inequality is strict if there exists $\Gamma \in \Omega^{ \mathcal{P}}(\mc{X}_s) \setminus \Omega^\mathrm{rat}_{\mathrm{int}}(\mc{X}_s)  $ such that $\Gamma(\ell_{\Gamma}) \neq \emptyset$. 

\smallskip

    For the second inequality, we observe that
    $\Omega(\mc{X}_s) \setminus \Omega^\mathrm{rat}_{\mathrm{int}}(\mc{X}_s)$ corresponds exactly to the set of $G$-orbits of singular cyan $G$-rigidities in $\mc{D}(\mc{X}'_s)$, by the second part of \Cref{rigidity-class}. 
    Hence, the second inequality follows from \Cref{border-line vorteces are rational}, with obvious strictness of the inequality if there exists $\Gamma \in  \Omega(\mc{X}_s)
    \setminus \left(\Omega^{\mathcal{P}}(\mc{X}_s) \cup \Omega^\mathrm{rat}_{\mathrm{int}}(\mc{X}_s)\right)$. 
     We also see easily from \Cref{border-line vorteces are rational} 
    that strictness of the second inequality holds, whenever $\mc{D}(\mc{X}'_s)$ contains a non-singular $G$-rigidity.
    So let us assume that $\mc{D}(\mc{X}'_s)$ does not contain a non-singular $G$-rigidity, and moreover that $ \Omega(\mc{X}_s) \setminus \Omega^\mathrm{rat}_{\mathrm{int}}(\mc{X}_s)  =  \Omega^{\mathcal{P}}(\mc{X}_s) \setminus \Omega^\mathrm{rat}_{\mathrm{int}}(\mc{X}_s)$.
    
    Let us consider the case  $\beta(\mc{D}(\mathcal{X}'_s)) \neq 1$. According to \Cref{border-line vorteces are rational}, supposing equality instead of inequality for the sake of contradiction, implies in this case  that $\ell_{\Gamma'} = k$ for all $\Gamma' \in \Omega(\mc{X}_s) \setminus \Omega^\mathrm{rat}_{\mathrm{int}}(\mc{X}_s)$.
    Recall that by the hypothesis of $(ii)$ 
  there is a finite extension $\ell/k$ that is $\neg\mathcal{P}$-minimal in $\mc{X}_s$, as well as $\Gamma, \widetilde{\Gamma} \in \Omega^{\mathrm{rat}}_{\mathrm{int}}(\mc{X}_s)$ together with 
$P \in \Gamma(\ell) \cap \widetilde{\Gamma}(\ell)$.
Let $Y \subseteq \mc{X}_s$ be a maximal connected closed subcurve containing $\Gamma$ and $\widetilde{\Gamma}$, such that every irreducible component $\Gamma'$ of $Y$ satisfies $\ell_{\Gamma'}\simeq \ell$, and such that $Y$ remains connected after removing all closed points $Q$ in which distinct irreducible components of $Y$ intersect with $\ell  \subsetneq \kappa(Q)$. So, $(i)$ and $(ii)$ of \Cref{rigidity-class} are automatically satisfied for $Y$.  Let $\Delta$ be an irreducible component of $\mc{X}_s$ such that $\Delta(\ell) \cap Y(\ell) \neq \emptyset$. In particular, $\Delta \notin \Omega^{\mathcal{P}}(\mc{X}_s)$. On the other hand, $\ell_\Delta \subseteq \ell$ and thus satisfies $\neg\mathcal{P}(\ell_\Delta)$. If the inclusion were proper, this would imply that $\Delta \notin \Omega^\mathrm{rat}_{\mathrm{int}}(\mc{X}_s)$, by $\neg\mathcal{P}$-minmality of $\ell$, but we are assuming that  $ \Omega(\mc{X}_s) \setminus \Omega^\mathrm{rat}_{\mathrm{int}}(\mc{X}_s)  =  \Omega^{ \mathcal{P}}(\mc{X}_s) \setminus \Omega^\mathrm{rat}_{\mathrm{int}}(\mc{X}_s)$, so  $\ell_\Delta = \ell$, and thus $\Gamma$ is a component of $Y$. However, by \Cref{rigidity-class}, this leads to the contradiction that $Y$ corresponds to a $G$-orbit of non-singular $G$-rigidities.

Under the previous assumptions, let us now consider the case $\beta(\mc{D}(\mathcal{X}'_s)) =1$. According to
 \Cref{border-line vorteces are rational}, supposing equality instead of inequality for the sake of contradiction, implies in this case that $\ell_{\Gamma'}\neq k$ for both $\Gamma'\in \Omega(\mc{X}_s) \setminus \Omega^\mathrm{rat}_{\mathrm{int}}(\mc{X}_s)$.  Let again $\ell/k$ be  $\neg\mathcal{P}$-minimal in $\mc{X}_s$, and $\Gamma \neq \widetilde{\Gamma} \in \Omega^\mathrm{rat}_{\mathrm{int}}(\mc{X}_s)$ that $\Gamma(\ell) \cap \widetilde{\Gamma}(\ell) \neq \emptyset$, as guaranteed by the hypothesis of $(ii)$, whereby $\ell_\Gamma = \ell_{\widetilde{\Gamma}}=\ell$. Let again  $Y \subset \mc{X}_s$ be a closed connected subcurve  that contains  $\Gamma$ and $\widetilde{\Gamma}$ and which is maximal with respect to the property that $\ell_{\Gamma'}= \ell$ for every irreducible component $\Gamma'$ of $Y$ and such that $Y$ remains connected after removing all intersection points between distinct components that are not $\ell$-rational. Since by assumption $Y$ cannot correspond to a non-singular $G$-rigidity, there must exist a irreducible component $\Delta$ of $\mc{X}_s$ such that $\Delta(\ell) \cap Y(\ell) \neq \emptyset$ but $\ell_\Delta \subsetneq \ell$. By the $\neg\mathcal{P}$-minimality of $\ell$, we conclude that $\Delta \notin \Omega^\mathrm{rat}_{\mathrm{int}}(\mc{X}_s)$, but on the other hand, we also have clearly that $\Delta \notin \Omega^{\mathcal{P}}(\mc{X}_s)$, which contradicts $ \Omega(\mc{X}_s) \setminus \Omega^\mathrm{rat}_{\mathrm{int}}(\mc{X}_s)  =  \Omega^{\mathcal{P}}(\mc{X}_s) \setminus \Omega^\mathrm{rat}_{\mathrm{int}}(\mc{X}_s)$.
\end{proof}

\begin{rem}
    In this article, the only meaningful application of \Cref{non-rat vorteces} is for either the empty property or for the property ``being non-real''. Nevertheless we choose to state it in the abstract framing, since on the one hand it potentially allows for  a more direct application in future articles, and on the other hand it helps to expose the necessity of certain subtleties in the definition of $G$-rigidity.
\end{rem}

\section{Arithmetic genus inequalities in residue characteristic zero}

Let us first recall the definition of the \textit{arithmetic genus for an algebraic curve $C$ over a field $\ell$}, which is defined as
$$\mathfrak{g}(C/\ell)=1 - \mathrm{dim}_\ell H^0(\mathcal{O}_C, C) +\mathrm{dim}_\ell H^1(\mathcal{O}_C, C).$$
We recall from \cite[Corollary 5.2.27]{Liu} that the arithmetic genus is stable under base change. 
We also recall that there is a notion of $g(E/\ell)$, the \textit{genus of a function field in one variable $E/\ell$} defined in \cite{Deu} based on valuation-divisors. Recall from \cite[Proposition 2.1]{BG24} that for the unique projective regular curve $C$ over the relative algebraic closure $\widetilde{\ell}$ of $\ell$ in $E$ such that $E \simeq_\ell \ell(C)$, we have
$$\mathfrak{g}(E/\ell)= \mathfrak{g}(C/\widetilde{\ell}).$$  

As in the previous section, we fix again a field $K$ of characteristic zero, $T \subseteq K$ a discrete valuation ring of $K$ with residue field $k$, a function field $F/K$ 
 in one variable such that $K$ is relatively algebraically closed in $F$, and $\mc{X}$ a regular model with normal crossing for $F$ over $T$,  a maximal unramified extension $T'$ of $T$ inside an algebraic closure of $K$, and we set $\mc{X}'=\mc{X} \times_T T'$. Throughout this section, we additionally assume $$\mathrm{char}(k)=0,$$
 which is crucial for the following:
\begin{prop}\label{reducing the curve}
$\beta(\mathcal{D}(\mc{X}_s)) \leq  \mathfrak{g}(\mc{X}_{s}/k) - \sum_{i=1}^n \mathfrak{g}(\Gamma_i/k).$
\end{prop}
\begin{proof}
Let $\mathcal{Y}$ be the special fiber  $\mathcal{X}_s$ with the induced reduced structure, that is with structure sheaf $\mc{O}_{\mc{Y}}= \mc{O}_{\mc{X}_s}/\mc{N}$, where $\mc{N}$ is the ideal sheaf of nilpotents in $\mc{O}_{\mc{X}_s}$. Then we have $\beta(\mc{D}(\mathcal{X}_s)) \leq  \mathfrak{g}(\mc{Y}/k) - \sum_i \mathfrak{g}(\Gamma_i/k)$ by \cite[Proposition 10.1.51]{Liu}. Furthermore, by a cohomolgical flatness result when $\mathrm{char}(k)=0$, see \cite[Proposition 6.4.2]{Ray70}, we have that  $H^0(\mc{X}_s, \mathcal{O}_{\mc{X}_s})=k$.
Then  $$\mathfrak{g}(\mc{Y}/k) = 1 - \mathrm{dim}_k  H^0(\mc{Y}, \mathcal{O}_{Y}) + \mathrm{dim}_k H^1(\mc{Y}, \mathcal{O}_{Y}) \leq \mathrm{dim}_k H^1(\mc{X}_s, \mathcal{O}_{\mc{X}_s}) =\mathfrak{g}(\mc{X}_{s}/k)$$
follows from the exact sequence of $k$-vector spaces $$0 \to 0 \to k \to H^0(\mc{Y}, \mc{O}_\mc{Y}) \to  H^1(\mc{X}_s, \mc{N})\to H^1(\mc{X}_s,\mc{O}_{\mc{X}_s})  \to H^1(\mc{Y},\mc{O}_{\mc{Y}}) \to 0$$ obtained by applying sheaf cohomology to the sequence 
$$ 0 \to \mc{N} \to \mc{O}_{\mc{X}_s} \to \mc{O}_\mc{Y} \to 0$$
of  $\mc{O}_{\mc{X}_s}$- modules.
\end{proof}
\begin{cor}\label{Betti-Genus}
$$\beta(\mathcal{D}(\mc{X'}_s)) \leq  \mathfrak{g}(F/K) - \sum_{w \in \Omega_T(F)} [\ell_w:k]\cdot \mathfrak{g}(\kappa_w/k).$$
\end{cor}

\begin{proof}
 Since $K$ is algebraically closed in $F$, we have that 
 $\mc{X}_0$ is geometrically integral, see \cite[Corollary 2.3.14]{Liu}.  Moreover, we have that $\mathfrak{g}(F/K)=\mathfrak{g}(\mc{X}_0/K)$, and by \cite[Corollary 8.3.6]{Liu}, we have that $\mathfrak{g}(\mc{X}_0/K)=\mathfrak{g}(\mc{X}_s/k)$.
 Since the arithmetic genus of a curve is stable under base change and since $\mc{X}_s \times_k k' \simeq \mc{X}'_s$, we have that $\mathfrak{g}(\mc{X}_s/k)=\mathfrak{g}(\mc{X}'_s/k')$, and hence $\mathfrak{g}(\mc{X}'_s/k')=\mathfrak{g}(F/K)$.
 On the other hand, every irreducible component of $\mc{X}'_s$ is a connected component of $\Gamma \times_k k'$ for an irreducible component $\Gamma$ of $\mc{X}_{s}$.
 More precisely $\Gamma \times_k k'$ is the disjoint union of $[\ell_\Gamma: k]$ copies of $\Gamma \times_{\ell_\Gamma} k'$. Note that $\mathfrak{g}(\Gamma \times_{\ell_\Gamma} k' /k')=\mathfrak{g}(\Gamma/\ell_\Gamma)=\mathfrak{g}(\kappa_w/k)$, where $w$ is the discrete valuation on $F$ induced by the generic point of $\Gamma$. Considering \Cref{algebrorational}, we obtain the inequality by applying the inequality of \Cref{reducing the curve} to the regular model with normal crossing $\mc{X}'$ of $F'=FK'$ over $T'$.
\end{proof}

\begin{rem}In \cite[Theorem 5.3]{BG24}, the previous inequality was obtained only for the minimal regular model, but where $k$ was only required to be perfect, instead of our standing assumption that $\car(k)=0$.
\end{rem}

The statements of the following Lemma is probably well known. 

\begin{lem}\label{rational and real specialization}
    \begin{enumerate}
    \item[a)] If $K$ admits a finite extension $L/K$ in which $T$ ramifies completely and such that $L$ contains the residue field of  a $K$-trivial valuation on $F$, then $\mc{X}_s(k) \neq \emptyset$.
    \item[b)] If $F$ is real and $T$ is henselian, then there exists a finite real extension $\ell/k$ such that $\mc{X}_s(\ell) \neq \emptyset$.
    \end{enumerate}
\end{lem}
\begin{proof}
The model $\mc{X}$ over $T$ can be defined by finitely homogeneous polynomials $f_1(X_0,\ldots,X_m), \ldots, f_r(X_0,\ldots,X_m)$ with coefficients in $T$. The special fiber of $\mc{X}_s$ is then defined by the polynomials
$\overline{f_1}(X_0,\ldots,X_m), \ldots, \overline{f_r}(X_0,\ldots,X_m)$ obtained by projecting the coefficients into the residue field $k$.
A $K$-rational place of $F$ corresponds to a $K$-rational point of the generic fiber $\mc{X}_0$, which in turn corresponds to a projective tuple $[x_0:\ldots :x_m] \in \mathbb{P}(K^m)$ solving the homogeneous equations given by the previous homogeneous polynomials. We may assume that all $x_i \in T$ and that without loss of generality $x_0=1$. Then $[\overline{x_0}:\ldots :\overline{x_m}] \in \mathbb{P}(k^m)$ defines a $k$-rational point in $\mc{X}_s$. This proves a).
Let us now assume  that $T$ is henselian and that $F$ is formally real. The latter implies in particular that $K$ is formally real, which implies that $k$ is formally real, by the assumption that $T$ is henselian. Moreover, $F$ being formally real is equivalent to the existence of an $L$-rational point of $\mc{X}_K$ for some finite real extension $L/K$, see for example \cite[Proposition 2.3]{G15}.
Since $T$ is henselian, there exists a unique valuation extension to $L$, and we denote its valuation ring by $T_L$. We denote its residue field by $\ell$, which is formally real, since $T_L$ is henselian.
If $[x_0:\ldots :x_m] \in \mathbb{P}(L^m)$ is a solution to the system of homogeneous equation, we may assume that all $x_i \in T_L$ and without loss of generality $x_0=1$. Hence, $[\overline{x_0}:\ldots :\overline{x_m}] \in \mathbb{P}(\ell^m)$, and it defines an $\ell$-point in $\mc{X}_s$. This proves b).
\end{proof}

We say that \textit{$F/K$ admits a totally $T$-ramified place} if there exists a valuation on $F$ that is trivial on $K$ such that its residue field (which is a finite extension of $K$) admits a unique, totally ramified, valuation ring extension of $T$. We are now in position to prove our two geometric arithmetic main theorems of this article.

\begin{thm}\label{rational-genus-inequality}
 $$\vert \Omega_{T}(F) \setminus \Omega^{\mathrm{rat}}_T(F) \vert + \sum_{w \in \Omega_T(F)} [\ell_w:k] \mathfrak{g}(\kappa_w/k)  \, \leq \, \mathfrak{g}(F/K) +1.$$
 Moreover, this inequality is strict if $F/K$ admits a totally $T$-ramified place.
\end{thm}
\begin{proof}
 Let $\Gamma_1, \ldots, \Gamma_n$ be the irreducible components of $\mc{X}_s$. Let $\mc{D}(\mc{X}_s)$ be the dual graph of $\mc{X}_s$ and $\beta$ its Betti number. Then we have that 
$$\beta \leq \mathfrak{g}(C/K) - \sum_{i=1}^n \mathfrak{g}(\Gamma_i/k)=\mathfrak{g}(F/K) - \sum_{i=1}^n [\ell_{\Gamma_i}:k] \mathfrak{g}(k(\Gamma_i)/k).$$
On the other hand, by \Cref{non-rat vorteces} applied to the empty property $\mathcal{P}$, we have that $$ \vert \Omega_{T}(F)  \setminus \Omega^{\mathrm{rat}}_{T}(F)  \vert \leq  \vert \Omega(\mc{X}_s) \setminus \Omega^\mathrm{rat}_{\mathrm{int}}(\mc{X}_s) \vert \leq \beta + 1.$$
If $F$ admits a $K$-rational place, or a place to an totally ramified extension of $K$ with respect to $T$, then $\mc{X}_s(k) \neq \emptyset$ by \Cref{rational and real specialization}.  Clearly, since all the irreducible components $\Gamma$ of $\mc{X}_s$ are smooth over $k$, we have that $\ell_\Gamma=k$ for every component $\Gamma$ with $\Gamma(k) \neq \emptyset$.
First consider the case, where there exists an irreducible component $\Gamma$ with $\Gamma(k)\neq \emptyset$ but $\Gamma$ intersects no other irreducible component of $\mc{X}_s$ in a $k$-rational point. Then $ \vert \Omega_T(F)  \setminus \Omega^{\mathrm{rat}}_T(F)  \vert <  \vert \Omega(\mc{X}_s) \setminus \Omega^\mathrm{rat}_{\mathrm{int}}(\mc{X}_s) \vert $, by \Cref{non-rat vorteces}.  
In the complementary case, we have that every irreducible component $\Gamma$ with $\Gamma(k) \neq \emptyset$ intersects at least one other irreducible component in a $k$-rational point. In particular, since $\mc{X}_s(k) \neq \emptyset$, there exists at least one such component.
Hence $\vert \Omega(\mc{X}_s) \setminus \Omega^\mathrm{rat}_{\mathrm{int}}(\mc{X}_s) \vert < \beta + 1$ in this case,  by \Cref{non-rat vorteces}.
\end{proof}

As already mentioned in the introduction, we set
$$\Omega^{\text{r}}_T(F):=\{w \in \Omega_T(F) \mid \kappa_w \text{ is  real }\},$$
as well as $$ \Omega^{\text{n/r}}_T(F):=\{w \in \Omega_T(F) \mid \kappa_w \text{ is nonreal and }  \overline{k}^{\mathrm{alg}} \cap \kappa_w  \text{ is real }\}.$$ 
Note that both sets are empty if $K$ is nonreal, and the first set is empty if $F$ is nonreal.
Since no nonreal function field with real field of constants can permit a rational place,  we obtain directly from \Cref{rational-genus-inequality} the following genus inequality:

\begin{cor}\label{nonreal-genus-inequality} 
 $$\sum_{w \in \Omega^{\text{n/r}}_T(F)} \!\!\!\! 1+ [\ell_w:k]\cdot \mathfrak{g}(\kappa_w/k) \,\, \leq  \,\,\mathfrak{g}(F/K) + 1.$$   
\end{cor}
This corollary is obviously only nontrivial when $F$ is nonreal.
We now prove a nontrivial genus inequality in the real case. We need to additionally assume henselianity of $T$. It should be noted already that due to optimality examples in the latter application to sums of squares in function fields, both inequalities (in the nonreal case and in the real case) turn out to be optimal.

\begin{thm}\label{real-genus-inequality} Suppose $T$  is henselian and $F$ is real. Then
$$\sum_{w \in \Omega^{\text{r}}_T(F)} \!\! [\ell_w:k] \cdot \mathfrak{g}(\kappa_w/k) \, \, + \sum_{w \in \Omega^{\text{n/r}}_T(F)} \!\!\!\! 1+ [\ell_w:k]\cdot \mathfrak{g}(\kappa_w/k) \,\, \leq  \,\,\mathfrak{g}(F/K).$$
\end{thm}
\begin{proof} Since $F$ is real, we have that so is $K$, and by the henselian hypothesis on $T$, we have in particular that $k$ is real.
Note first that if $w \in \Omega^{n/r}_T(F)$ then in particular $w \in \Omega^{\mathcal{P}}_T(F) \setminus  \Omega^{\mathrm{rat}}_T(F)$ for the field property $\mathcal{P}=\text{``is nonreal''}$, since if the function field $\kappa_w$ had a rational place, then it could not possibly be non-real.
Hence, abbreviating $\delta_w=[\ell_w:k]$ for $w\in \Omega_T(F)$, we have that
 \begin{align*} \sum_{w \in \Omega^{\text{r}}_T(F)} \!\!\!\delta_w \mathfrak{g}(\kappa_w/k)+ \!\!\!\!\!\sum_{w \in \Omega^{\text{n/r}}_T(F)}\!\!\! 1+ \delta_w  \mathfrak{g}(\kappa_w/k) 
  &\leq \vert \Omega^{\mathcal{P}}_{T}(F) \setminus \Omega^{\mathrm{rat}}_T(F) \vert + \!\!\! \sum_{w \in \Omega_T(F)} \delta_w \mathfrak{g}(\kappa_w/k) \\
  &\leq  \vert \Omega^{\mathcal{P}}(\mc{X}_s) \setminus \Omega^\mathrm{rat}_{\mathrm{int}}(\mc{X}_s) \vert + \!\!\! \sum_{w \in \Omega_T(F)} \delta_w\mathfrak{g}(\kappa_w/k) \\
  &\leq \mathfrak{g}(F/K) + 1.
  \end{align*}
  for both the empy field property $\mathcal{P}=\text{``$1\neq 0$''}$ and the field property $\mathcal{P}=\text{``is nonreal''}$, which both satisfy going up. The first inequality is directly obvious, since for $w \in \Omega^{n/r}_T(F)$ and $\Gamma \in \Omega(\mc{X}_s)$ such that $w=v_{\Gamma}$, we obviously have that $\Gamma \in \Omega^{\mathcal{P}}(\mc{X}_s) \setminus \Omega^\mathrm{rat}_{\mathrm{int}}(\mc{X}_s)$. The second and third inequality follow from \Cref{non-rat vorteces} and \Cref{Betti-Genus}.
By \Cref{rational and real specialization} there exists a finite real extension $\ell/k$ such that  $\mathcal{X}_s(\ell) \neq \emptyset$. Let us choose $\ell/k$ of minimal degree with this property.
Let $\Gamma \subseteq \mc{X}_s$ be an irreducible component such that $\Gamma(\ell)\neq \emptyset$. In particular $\ell_\Gamma \subseteq \ell$, whereby $\ell_\Gamma$ is real.

\smallskip

If $\ell_\Gamma \subsetneq \ell$, then $v_\Gamma \in \Omega^{\text{r}}_T(F)$ and $\Gamma \notin \Omega^\mathrm{rat}_T(F)$ by the minimality assumption on $\ell$, whereby the first of the previous inequalities is obviously strict: $$\sum_{w \in \Omega^{\text{r}}_T(F)} \!\! \delta_w\mathfrak{g}(\kappa_w/k)+ \sum_{w \in \Omega^{\text{n/r}}_T(F)}\!\!\! 1+ \delta_w\mathfrak{g}(\kappa_w/k) < \vert \Omega_{T}(F) \setminus \Omega^{\mathrm{rat}}_T(F) \vert + \sum_{w \in \Omega_T(F)} \delta_w \mathfrak{g}(\kappa_w/k).$$

If $\ell_\Gamma = \ell$ and $\Gamma \notin \Omega^\mathrm{rat}_{\mathrm{int}}(\mc{X}_s)$, then $$\vert \Omega_T(F) \setminus \Omega^\mathrm{rat}_T(F) \vert < \vert \Omega(\mc{X}_s) \setminus \Omega^\mathrm{rat}_{\mathrm{int}}(\mc{X}_s) \vert.$$

Suppose now that $\ell_\Gamma = \ell$ and $\Gamma \in \Omega^\mathrm{rat}_{\mathrm{int}}(\mc{X}_s)$. Hence, there exists $\widetilde{\Gamma} \in \Omega(\mc{X}_s)$ such that 
$\kappa(P) =\ell$ for some $P \in \Gamma \cap \widetilde{\Gamma}$. In particular $v_{\wt{\Gamma}} \in \Omega^{r}_T(F)$. If $\ell_{\wt{\Gamma}} \subsetneq \ell$, then we have by  minimality of $\ell/k$ in particular that 
  $\widetilde{\Gamma} \notin\Omega^\mathrm{rat}_{\mathrm{int}} (\mc{X}_s)$, whereby, as before, we have strictness in the first inequality: $$\sum_{w \in \Omega^{\text{r}}_T(F)}\!\! \delta_w\mathfrak{g}(\kappa_w/k)+ \sum_{w \in \Omega^{\text{n/r}}_T(F)}\!\!\! 1+ \delta_w\mathfrak{g}(\kappa_w/k) < \vert \Omega_{T}(F) \setminus \Omega^{\mathrm{rat}}_T(F) \vert + \sum_{w \in \Omega_T(F)} \delta_w \mathfrak{g}(\kappa_w/k).$$
If, on the other hand, $\ell_{\wt{\Gamma}} = \ell$, then 
we have $$\vert \Omega^{\mathcal{P}}(\mc{X}_s) \setminus \Omega^\mathrm{rat}_{\mathrm{int}}(\mc{X}_s) \vert + \sum_{w \in \Omega_T(F)} \!\! [\ell_w:k] \mathfrak{g}(\kappa_w/k) < \mathfrak{g}(F/K) + 1$$
 by $(ii)$ of \Cref{non-rat vorteces} applied to the property $\mathcal{P}=\text{``is nonreal''}$, in view of \Cref{Betti-Genus}.
 In either case, we obtain a strict inequality for one of the intermediate inequalities, whereby we may replace $\mathfrak{g}(F/K) + 1$ by $\mathfrak{g}(F/K)$ on the right hand side of the first chain of weak inequalities in this proof.
\end{proof}

\section{A bound on sums of squares}\label{sumsofsquaressection}

Let $K$ be a field of characteristic different from $2$. For $m \in \mathbb{N}$, we denote by
 $\sum^m K^2$ the subset of $K$ consisting of elements that can be written as a sum of $m$ squares. We write $\sum K^2$ for the set of all elements that can be written as a finite sum of squares. The
\textit{Pythagoras number of $K$} is the smallest positive integer $m$ such that $\sum^m K^2 = \sum K^2$ if such an $m \in \mathbb{N}$ exists, and $\infty$ otherwise. We denote it by $p(K)$. The \textit{level of $K$} is the smallest positive integer $m$ such that $-1 \in  \sum^m K^2$ when $K$ is nonreal, and $\infty$ if $K$ is real. We denote it by $s(K)$. When $K$ is nonreal, we have the pythagoras-level inequality 
$$  s(K)\leq p(K)\leq s(K)+1, $$
due to the identity $x=(\frac{x+1}{2})^2-(\frac{x-1}{2})^2.$ 

We recall that $\sum^m K^2$ is multiplicatively closed when $m$ is a power of two and that $s(K)$ is always a power of two when $F$ is nonreal, as shown by A. Pfister, see \cite[Theorem 1.1, Theorem 2.2]{Lam}. For $m\in\mathbb{N},$ we denote by $\big(\sum^{2^\ell} \!\! K^2 \big)^\times\!\!$, respectively by $(\sum K^2)^\times$, the multiplicative group of nonzero sums of $2^\ell$ squares, respectively of an arbitrary finite number of squares, in $K$. We set $K^{\times 2}:=\left(\sum^{1}K^{2}\right)^{\times}.$ Considering the filtration of groups
$$ K^{\times 2} \subseteq \left(\sum^{2} K^2\right)^\times \subseteq \ldots \subseteq \left(\sum^{2^\ell} K^2\right)^\times\subseteq \ldots \subseteq \left(\sum^{\phantom{\infty}} K^2\right)^\times \subseteq K^\times,$$
the size of successive the quotient groups in this filtration gives additional information on the sum of squares structure of $K$ beyond the pythagoras number of the field. Clearly, the inclusions become stationary at the smallest $2$-power bigger or equal than the Pythagoras number of $K$.

For $\ell\in\mathbb{N}$, we let $G_{\ell}(K)=\left(\sum^{2^{\ell+1}} K^2\right)^\times/ \left(\sum^{2^{\ell}}  K^2\right)^{\times}.$ We define the \textit{$\ell^{\text{th}}$ Pfister index of $K$} as $$\rho_\ell(K):=\mathrm{log}_2 
 \left\vert G_{\ell}(K)\right\vert \in \mathbb{N} \cup \{\infty\}.$$

\begin{lem}\label{bound-henselian}
 Let $\ell\in \mathbb{N}.$ Let $K$ be a field with a discrete henselian valuation $w$ with $\mathrm{char}(\kappa_w) \neq 2$. Then
$$\rho_\ell(K) \leq \left\{ \begin{array}{rl}\rho_\ell(\kappa_w), & \text{ if $s(\kappa_{w}) \geq 2^{\ell+1}$,} \\  1+ \rho_\ell(\kappa_w), & \text{ if } s(\kappa_w)=2^{\ell}. \end{array} \right.$$ Moreover, if $\ell\geq 1$ and either $s(\kappa_{w})\leq 2^{\ell-1}$ or $s(\kappa_w)= \infty$ and $p(\kappa_w)\leq 2^\ell$ then $\rho_{\ell}(K)=0.$ 
\end{lem}
\begin{proof} We assume first that $s(\kappa_{w})\geq 2^{\ell}.$ Note that $2^\ell \leq p(\kappa_{w})\leq 2^{\ell}+1\leq 2^{\ell+1}$ by the pythagoras-level inequality when $s(\kappa_{w})=2^{\ell},$ and that $w(\sigma)\in 2\mathbb{Z}$ for all $\sigma\in \left(\sum^{2^{\ell+1}} K^{2}\right)^\times$ when $s(\kappa_{w})\geq 2^{\ell+1}$ (see for example \cite[Lemma 4.1]{BGVG14}). We write $\overline{n}$ for $n+2\mathbb{Z}$ for $n \in \mathbb{Z}$ and $\overline{x}$ for the residue in $\kappa_w$ for any $x\in K$ with $w(x)=0$.
Fix any $t\in K$ with $w(t)=1$. Consider
 the group homomorphism
\begin{align*} \Phi:G_{\ell}(K) \to & \mathbb{Z}/2\mathbb{Z}\times G_{\ell}(\kappa_{w}) \\
[\sigma] \mapsto &  \left( \overline{w(\sigma)}, [\overline{\sigma t^{-w(\sigma)}}] \right)\end{align*} For $[\sigma] \in \mathrm{ker}(\Phi)$, we have that  $\overline{\sigma t^{-w(\sigma)}}\in \left(\sum^{2^{\ell}}\kappa_{w}^{2}\right)^\times,$ and hence $\sigma\in \sum^{2^{\ell}}K^{2}$, since $w$ is henselian. This shows that $\Phi$ is injective, and hence $\rho_{\ell}(K)\leq 1+\rho_{\ell}(\kappa_{w}).$ If $s(\kappa_{w})\geq 2^{\ell+1},$ then by the earlier argument, we have that $\mathrm{im}(\Phi) \subseteq \{0\} \times G_{\ell}(K)$, and hence $\rho_{\ell}(K)\leq \rho_{\ell}(\kappa_{w})$. Finally,  if $p(\kappa_{w})\leq 2^{\ell}$ then $\rho_{\ell}(\kappa_{w})=0$. If, in this situation, $s(\kappa_w)=\infty$, then by what we have already shown, we obtain $\rho_\ell(K)= \rho_\ell(\kappa_w)=0$, and if on the other hand $s(\kappa_w) \leq 2^{\ell-1}$ then $s(K)\leq 2^{\ell-1}$, and thus $p(K) \leq 2^\ell$, whereby $\rho_\ell(K)=0$.
\end{proof}

 Let $K$ be the field of fractions of a discrete valuation ring $T$ with residue field $k$ of characteristic different from $2$, and let $F/K$ be a function field. For $\ell\in\mathbb{N}$ we set $$\Omega^{\ell}_T(F)=\{w\in\Omega_T (F)\mid 2^\ell = s(\kappa_{w})\}$$
 and 
 $$\Omega^{>\ell}_T(F)=\{w\in\Omega_T (F)\mid 2^{\ell+1} \leq s(\kappa_{w})\}.$$

\begin{lem}\label{lema-cotas-rho} Let $\ell\in\mathbb{N}$  such that $\ell\geq 1.$ Let $K$ be the field of fractions of a nondyadic discrete henselian valuation ring $T$ with residue field $k$ of characteristic different from $2$. Assume that $p(k(X))\leq 2^{\ell}.$ Let $F/K$ be a function field in one variable. Then $$\rho_{\ell}(F)\leq \sum_{w\in\Omega_T^{> \ell}(F)}\rho_{\ell}(\kappa_{w})+\sum_{w\in\Omega^{\ell}_T(F)}1+\rho_{\ell}(\kappa_{w}).$$ 
\end{lem}
\begin{proof} We recall the local-global principle for isotropy of quadratic forms at least $3$ variables over $F$ with respect to all discrete valuations in $F$ in case that $T$ is complete from \cite[Theorem 3.1]{CTPS12}. The completeness condition can be relaxed to the henselian condition, see \cite[Theorem 4.4]{BDGMZ}.
For a discrete valuation $w$ on $F,$ we denote by $F^{w}$ its completion with respect to $w$. Let $\Omega(F)$ denote the set of all discrete valuations on $F$. Applied to our situation, the local-global principle implies the injectivity of the natural homomorphism $$G_{\ell}(F)\hookrightarrow \prod_{w\in\Omega(F)}G_{\ell}(F^{w}),$$
and thus 
$\rho_\ell(F) \leq \sum_{w \in \Omega(F)} \rho_\ell(F^w)$,
since if $f \in F$ is any sum of $2^{\ell+1}$ squares that becomes locally a sum of $2^\ell$ squares, then
the quadratic form $X^2_1 + \ldots + X^2_{2^\ell} - f Y^2$
is certainly isotropic over $F$ since $\ell \geq 1$, by the local-global principle, and thus $f$ is already a sum of $2^\ell$ squares in $F$.
Observe that by \cite[Theorem 3.5]{BVG}, since $p(k(X))\leq 2^\ell$, we have for any $w\in \Omega(F)$ such that $w(K)=0$ that $s(\kappa_w) \leq 2^{\ell-1}$ if  $\kappa_w$ is nonreal, as well as $p(\kappa_w)< 2^\ell$ if $\kappa_w$ is real. In particular $\rho_\ell(F^w)=0$, by \Cref{bound-henselian}. If $w(K) \neq 0$, then necessarily $\mathcal{O}_w \cap K=T$. Moreover, if $\kappa_w$ is algebraic over $k$, then $p(\kappa_w) < 2^\ell$, again by \cite[Theorem 3.5]{BVG} applied to a description of $\kappa_w$ as a direct limit of finite extensions of $k$.
Thus we only need to take into consideration residually transcendental valuation extensions of $T$:
$$\rho_\ell(F) \leq \sum_{w \in \Omega_T(F)} \rho_\ell(F^w)\leq \sum_{w \in \Omega^{>\ell}_T(F)} \rho_\ell(\kappa_w) + \sum_{w \in \Omega^\ell_T(F)} 1 + \rho_\ell(\kappa_w).$$
\end{proof}

Before stating and proving our main theorem in this section, 
we need to define some technical notions for valuations on $K$, which in the most important example situation  $K=\mathbb{R}(\!(t_1)\!)\ldots(\!(t_{n})\!)$ are just straight forward.

We say that a valuation  $v: K^\times \to (\Gamma,\leq)$ (without loss of generality $\Gamma=v(K^\times)$) is an \textit{$n$-discrete valuation on $K$} if $(\Gamma,\leq)$ is isomorphic as an ordered abelian group to the lexicographically ordered group $(\mathbb{Z}^{n},\leq_{\mathsf{lex}})$. 
For simplicity, let us identify $\Gamma= \mathbb{Z}^n$. Letting $\pi: \mathbb{Z}^n \to \mathbb{Z}$ denote the projection onto the lexicographically dominant component, we call the valuation $\pi \circ v: K^\times \to \mathbb{Z}$ the \textit{$1$-discrete coarsening of $v$}, and we denote it by $v_1$. 
On $\kappa_{v_1}$ we now have naturally defined an $(n-1)$-discrete valuation 
\begin{align*} 
\overline{v}: \kappa_{v_1}^\times  & \to \mathbb{Z}^{n} \\
                        \overline{x} & \mapsto v(x)
\end{align*}
that we call \text{the residual valuation of $v$ modulo its $1$-discrete coarsening}.
It is an easy exercise left to the reader to see that the dominant component of $v(x)$ is zero, and that indeed $v(x)$ does not depend on the choice of representant of $\overline{x}$.  We have that  that $\kappa_{v}=\kappa_{\overline{v}}$, see  \cite[Section 8]{EP05} . If $v$ is henselian, then so is $v_1$ and $\overline{v}$, see \cite[Corollary 4.1.4]{EP05}.  
One can also do this in reverse: Given a $1$-discrete valuation $v: K^\times \to \mathbb{Z}$ and an $n-1$-discrete valuation $w: \kappa_{v} \to \mathbb{Z}^{n-1}$, then we can define the composite valuation with respect to uniformizer $\pi$ of $v$ 
\begin{align*} 
w \circ v: K^\times  &\to \mathbb{Z} \times \mathbb{Z}^{n-1}  \\
    x &\mapsto (v(x), w(\overline{ x \pi^{-v(x)}} ) ).
\end{align*}  
We leave it again to the reader to verify that this is an $n$-discrete valuation, and that a different choice of uniformizer does not change the equivalence class of valuation that we obtain.

\begin{thm}\label{sumsofsquaresbound} Suppose that $K$ carries a henselian $n$-discrete valuation for some $n\in \mathbb{N}$ with residue field $k$ of characteristic zero. Suppose that  $p(E)\leq 2^{\ell}$ for every function field in one variable $E/k$ for some integer $\ell\geq 1$.  
Then $p(K(X)) \leq 2^\ell$, and for any function field $F/K$ in one variable, we have $p(F) \leq 2^\ell + 1$ with  $$\rho_{\ell}(F) \leq n\cdot (\mathfrak{g}(F/K)+1),$$ if $F$ is nonreal, and  $$\rho_{\ell}(F) \leq n\cdot \mathfrak{g}(F/K),$$
if $F$ is real.
\end{thm}
\begin{proof} The fact that $p(F) \leq 2^\ell + 1 $  for any function field $F/K$ was shown in \cite[Corollary 1.5]{BDGMZ}. 
We show  the remaining affirmation by induction on $n.$ If $n=0,$ then $k=K,$ so that by hypothesis $p(F)\leq 2^{\ell},$ that is $\rho_{\ell}(F)=0$ for every function field in one variable $F/K$. Now suppose that $n\geq 1$. Let $v$ be the henselian $n$-discrete valuation on $K$ with the properties given in the statement. Let $v_1$ be its $1$-discrete coarsening and let $T\subseteq K$ denote the discrete valuation ring of $v_1$ and $k_1$ its residue field. Denote by $\overline{v}$ the residual valuation of $v$ modulo $v_1$ on $k_1$. Since $\overline{v}$ is a $(n-1)$-discrete henselian valuation on $k_1$ with residue field $\kappa_{\overline{v}}=k$, we may apply the induction hypothesis to any function field in one variable over $k_1$. 
Let $F/K$ be a function field in one variable of genus $g$. Since the genus of a function field is defined with respect to the relative algebraic closure of the base field, we may assume without loss of generality that $K$ is algebraically closed in $F$, as any $n$-discrete henselian valuation  extends uniquely  to an $n$-discrete henselian valuation in a finite extension.   
We first observe that under the conditions of the theorem, we have $\Omega^{r}_T(F) \subseteq \Omega^{>\ell}_T(F)$
and 
$$ \left(\Omega^{>\ell}_T(F) \cup \Omega^\ell_T(F)\right) \subseteq \left(\Omega^{r}_T(F) \cup \Omega^{n/r}_T(F) \right),
$$
since the residue field $\kappa_w$ of any residually transcendental valuation extension $w$ of $T$ on $F$ with nonreal relative algebraic closure of $k_1$ in $\kappa_w$  has level at most $2^{\ell-1}$ by \cite[Theorem 3.5]{BVG}, since $p(k_1(X)) \leq 2^\ell$ by induction hypothesis. Moreover we conclude with  
\cite[Theorem 6.10]{BGVG14} that $p(K(X))\leq 2^{\ell}$ from this induction hypothesis. 

 Now, by \Cref{lema-cotas-rho}, the above mentioned inclusions of valuation sets, the induction hypotheses for $\rho_\ell$ for function fields over $k_1$, and either \Cref{nonreal-genus-inequality} or \Cref{real-genus-inequality} we obtain   
\begin{align*} \rho_{\ell}(F) &\leq \sum_{w\in\Omega_T^{> \ell}(F)}\!\!\rho_{\ell}(\kappa_{w})+\!\!\!\!\!\sum_{w\in\Omega^{\ell}_T(F)} \!\!1+\rho_{\ell}(\kappa_{w}) \, \leq \, \sum_{w\in\Omega_T^{\text{r}}(F)} \rho_{\ell}(\kappa_{w}) + \sum_{w\in\Omega_T^{\text{n/r}}(F)} 1+ \rho_{\ell}(\kappa_{w}) \\ 
&\leq \sum_{w\in\Omega_T^{\text{r}}(F)} \mathfrak{g}(\kappa_{w}/k_1)(n-1) +\sum_{w\in\Omega_T^{\text{n/r}}(F)} 1 +  (\mathfrak{g}(\kappa_{w}/k_1)+1)(n-1) \\ &\leq (n-1) \mathfrak{g}^{*}(F/K)  + \vert \Omega_T^{\text{n/r}}(F) \vert \;\; \leq \;\; n\cdot \mathfrak{g}^{*}(F/K),
\end{align*} 
where $\mathfrak{g}^{*}(F/K):= \mathfrak{g}(F/K)+ 1$ when $F$ is nonreal, and $\mathfrak{g}^{*}(F/K):= \mathfrak{g}(F/K)$ when $F$ is real, to  keep the inequalities visually readable. Note that we applied \Cref{nonreal-genus-inequality}, respectively \Cref{real-genus-inequality} in the penultimate, as well as the ultimate inequality. 
\end{proof}

\begin{ex}\label{mainsostheoremapplied} Let $n,\ell\in\mathbb{N}$ with $\ell \geq 1.$ It is well known that the pythagoras number of any function field in one variable over $\mathbb{R}(X_{1},\ldots,X_{\ell-1})$ is at most $2^\ell$,  see \cite[Example 1.4 (4), chap. 7]{Pf95}. Hence, the field $K=\mathbb{R}(X_{1},\ldots,X_{\ell-1})(\!(t_{1})\!)\ldots(\!(t_{n})\!)$ satisfies the hypothesis of \Cref{sumsofsquaresbound} and we obtain 
$$\rho_{\ell}(F)\leq n\cdot (\mathfrak{g}(F/K)+1)$$ for any function field $F/K$ in one variable, and 
$\rho_{\ell}(F)\leq n \cdot \mathfrak{g}(F/K)$ when $F$ is real. The case $\ell=1$ is the case discussed in the introduction.
\end{ex}


The optimality of the bound $\rho_1(F) \leq n\cdot (g+1)$ from \Cref{sumsofsquaresbound} in the nonreal case is a consequence of \cite[Example 4.5]{GM24}:  For $g\in\mathbb{N},$ the curve $$Y^{2}=-\prod_{i=0}^{g}(X^{2}+t_{n}^{2i})$$ over $\mathbb{R}(\!(t_{1})\!)\ldots(\!(t_{n})\!)$ has genus $g$, its function field is nonreal, and it was shown that $\rho_{1}(F)=n(g+1).$ The following example shows optimality of the bound $\rho_{1}(F) \leq n\cdot g$ in the case where $F$ is real:

\begin{ex}\label{real optimal example} Let $n,g\in\mathbb{N}.$ Let $F$ be the function field of the curve $$Y^{2}=(X-1)\displaystyle\prod_{i=1}^{g}(X^{2}+t_{n}^{2i}),$$ over $K=\mathbb{R}(\!(t_{1})\!)\ldots(\!(t_{n})\!).$ This curve has genus $g$ and $F$ is real. We will show that $\rho_{1}(F)=ng.$ 
 Let $v_n$ be the unique discrete valuation on $K$ (the $t_n$-adic one). Its residue field is $k=\mathbb{R}(\!(t_1)\!)\ldots (\!(t_{n-1})\!)$. For $1\leq j\leq g,$ let $w'_{j}$ be the so called Gauss extension of $v_n$ with respect to $Z=Xt_{n}^{-j},$ and let $w_{j}$ denote an extension of $w'_{j}$ to $F$. Note that $K(X)=K(Z).$ Since $F=K(Z)(\sqrt{g(Z)})$ with $$g(Z)=(t_{n}^{j}Z-1)(t_{n}^{2j-2}Z^{2}+1)\cdots(t^2Z^{2}+1)\cdot(Z^{2}+1)\cdot(Z^{2}+t^2)\cdots(Z^{2}+t_{n}^{2(g-j)}) \in \mathcal{O}_{w'_j},$$ 
we obtain that $$\kappa_{w_{j}}=\kappa_{v_n}(\overline{Z})\left(\sqrt{-\overline{g(Z)}}\right)\simeq \mathbb{R}(\!(t_1)\!)\ldots(\!(t_{n-1})\!)(X)\left(\sqrt{-(X^{2}+1)}\right).$$
Thus $s(\kappa_{w_{j}})=2.$ Since $k$ admits a henselian $n-1$-discrete valuation and $\kappa_{w_{j}}/k$ is a function field of genus zero, we have by \cite[Lemma 4.4]{GM24}  that $\rho_{1}(\kappa_{w_{j}})=n-1.$   By \cite[Theorem 3.6]{GM24} , we conclude that there must be $n-1$ valuations on $\kappa_{w_{j}}$
that are  $d$-discrete  for some $1\leq d\leq n-1$ and whose residue fields have level $2$. After composition with the respective $w_j$, we obtain for each $j \leq g$ a set of
 $n-1$ valuations on $F$ which are $d$-discrete for some  $2 \leq d\leq n$,  and since the composition does not change the residue field by \cite[Remark 2.3]{GM24}, they are of level $2$.
Counting additionally each $w_j$ as a $1$-discrete valuation, we obtain this way in total $gn$ distinct (nonequivalent) valuations on $F$ that are $d$-discrete for some $1 \leq d\leq n$  and with residue field of level $2$. Hence, \cite[Theorem 3.6]{GM24}  shows that $\rho_{1}(F)\geq ng$.
Of course, we have equality by \Cref{sumsofsquaresbound}.

\end{ex}

\begin{qu} Let $K=\mathbb{R}(X)(\!(t_1)\!)\ldots(\!(t_{n})\!)$. Is there a nonreal function field in one variable $F/K$ such that $\rho_2(F)=n\cdot (\mathfrak{g}(F/K) +1)$ ? 
Is there a real function field $F/K$ such that $\rho_2(F)=n\cdot\mathfrak{g}(F/K)$ ? If not, what is the correct formula for the optimal bound on $\rho_2$ ?
\end{qu}

We end this section by analyzing why the genus inequality from \cite{BG24} would not have sufficed to prove the bound $\rho_1(F) \leq n \cdot (\mathfrak{g}(F/K)+1)$. In our notation, the genus inequality shown in \cite{BG24} is 
$$\vert \Omega^0_{T}(F) \setminus \Omega^{\mathrm{rat}}_T(F) \vert + \sum_{w \in \Omega_T(F)} [\ell_w:k] \mathfrak{g}(\kappa_w/k)  \, \leq \, \mathfrak{g}(F/K) +1,$$
where $\Omega^0_{T}(F)=\{w \in \Omega_{T}(F)\mid \mathfrak{g}(\kappa_w/k)=0\}$, and a condition for when equality can occur is also given, which we do not repeat here, but suffice to say that this condition is taken into account in the following:

\smallskip

Let $K=\mathbb{R}(\!(t_1)\!)(\!(t_{2})\!)$.  Let us consider the situation where $F/K$ is a nonreal function field of genus $3$. Note that, in difference to the almost similar looking \Cref{nonreal-genus-inequality}, the above mentioned genus inequality does not exclude the possibility of the existence of $3$ residually transcendental valuation extensions $w_1,w_2$ and $w_3$ of the $t_2$-adic valuations on $K$, such that all $\kappa_{w_i}$ have genus $1$ over $\mathbb{R}(\!(t_1)\!)$ and the latter is relative algebraically closed in $\kappa_{w_i}$.    
Again, for each of  $\kappa_{w_i}/\mathbb{R}(\!(t_1)\!)$, there is no obstruction coming from the genus inequality above that would prevent the existence of two residually transcendental valuation extensions $v_{i,1}$ and $v_{i,2}$ of the $t_1$-adic one to $\kappa_{w_i}$, such that $\kappa_{v_{i,j}}$ are of genus zero over $\mathbb{R}$ and such that $\mathbb{R}$ is relatively algebraically closed in $\kappa_{v_{i,j}}$. In particular $\kappa_{v_{i,j}}\simeq \mathbb{R}(X)(\sqrt{-(X^2+1)})$, which has level $2$. This would force the existence of in total $6$ distinct $2$-discrete valuations with residue field of level $2$ on $F$. By \cite[Lemma 4.4]{GM24}, this implies that $\rho_1(F) \geq  2 \cdot 6= 12$. But we know from our previous result that $\rho_1(F) \leq 2 (\mathfrak{g}(F/K)+1)= 8$. 

\section{A bound on Local squares}
In this section, we assume that $K$ is the field of fractions of a complete non-dyadic discrete valuation ring $T$. Let $k$ denote its residue field. 
Let $F/K$ be a function field in one variable.  
Let $\mc{L}(F)$ denote the subgroup of $F^\times$ consisting of elements that are squares in the completion  $F^w$ for every $\zz$-valuation $w$ on $F$.

\smallskip
\noindent 
Let $\Sh(F)$ denote the kernel of the product restriction morphism of Witt groups
$$W(F) \to \prod_{w \in \Omega(F)} W(F^w).$$
We refer to \cite{Lam} for details about the construction and properties of Witt groups of quadratic forms over fields.

\begin{lem}\label{radicalker} 
$\faktor{\mathcal{L}(F)}{F^{\times 2}}\simeq \Sh(F).$
\end{lem}
\begin{proof}  
By the local-global principle for quadratic forms in $3$ or more variables with respect to $\Omega_T(F)$, any nontrivial Witt class in $\Sh(F)$ is defined by an anisotropic quadratic form of dimension $2$. Again, using this local-global principle, we can show that any scalar multiple of an anisotropic quadratic form of dimension two whose Witt class lies in $\Sh(F)$, defines the same Witt class.
So, any element in $\Sh(F)$ is represented by an anisotropic quadratic form $X^2 - c Y^2$ for some $c \in \mathcal{L}(F)$, and in fact two such $c$ define the same Witt class if and only if their product is a square in $F$.
\end{proof}

\begin{prop}\label{HHKbetti} Let $\mc{X}$ be a regular model with normal crossings for $F/T.$ Let $\beta$ be the Betti number of the dual graph $\mc{D}(\mc{X}_s)$. Then $$\left\vert \faktor{\mathcal{L}(F)}{F^{\times 2}} \right\vert =2^{\beta}.$$
\end{prop}
\begin{proof} 
One easily verifies that, since $\mc{X}_s$ has only normal crossings in $\mc{X}$, our reduction graph $\mathcal{D}(\mc{X}_s)$ coincides with the reduction graph defined in \cite{HHK15}.
For a point $x\in\mc{X}_{s},$ we denote by $F_{x}$ the field of fractions of the completion of the local ring $\mathcal{O}_{\mc{X},x}$ with respect to its maximal ideal. By \cite[Theorem 9.6]{HHK15}, the kernel of the product of restriction morphisms $\varphi: WF\rightarrow\displaystyle\prod_{x\in \mc{X}_{s}} WF_{x}$ is isomorphic to the abelian $2$-group $\mathsf{Hom}(\pi_{1}(\mc{D}(\mc{X}_s)),\mathbb{Z}/2\mathbb{Z}),$ where $\pi_{1}(\mc{D}(\mc{X}_s))$ is the fundamental group of the graph $\mc{D}(\mc{X}_s)$ as a topological space. It follows by \cite[Proposition 1A.1 and Proposition 1A.2]{Ha00} that the group $\pi_{1}(\mc{D}(\mc{X}))$ is freely generated by $\beta$ elements. Hence $\mathsf{Hom}(\pi_{1}(\mc{D}(\mc{X}_s)),\mathbb{Z}/2\mathbb{Z})$ is isomorphic to $(\mathbb{Z}/2\mathbb{Z})^{\beta},$ whereby $\vert \ker\varphi \vert = 2^\beta$. On the other hand, we have by \cite[Proposition 9.11 (c)]{HHK15} that $\ker(\varphi) = \Sh(F)$, and finally $\Sh(F)=\mathcal{L}(F)/F^{\times 2}$ by the previous Lemma. 
\end{proof}

Our final goal in this article is to establish  the following inequality:

\begin{thm}\label{boundlocal}
Assume that $\mathsf{char}(k)=0$. Then
$$\log_2\left|\faktor{\mathcal{L}(F)}{F^{\times 2}}\right|\leq \mathfrak{g}(F/K) - \sum_{w \in \Omega_T(F)} [\ell_w:k]\cdot \mathfrak{g}(\kappa_w/k).$$
\end{thm}
\begin{proof}
Let $\mc{X}$ be a regular model for $F/T$ whose special fiber is a normal crossing divisor. Let $T'$ be a maximal unramified valuation ring extension of $T$ inside an algebraic closure of $K$. Denote by $K'$ and $k'$ the  field of fraction and residue field of $T'$ respectively. We denote 
$\mc{X}':= \mc{X} \times_T T'$. It follows from \Cref{bettiinequality} that $\beta(\mc{D}(\mc{X}_s))\leq \beta(\mc{D}(\mc{X'}_s)).$
The statement now follows directly from \Cref{HHKbetti} with \Cref{radicalker} and \Cref{Betti-Genus}. 
\end{proof}

\begin{rem}
The assumption of residue characteristic zero is only needed in the proof for the inequality of \Cref{Betti-Genus}, whose proof uses a cohomological flatness result in the case of residue characteristic zero. However, it is feasible that said inequality holds regardless of cohomological flatness, and so \Cref{boundlocal} might hold without the assumption of residue characteristic zero.
\end{rem}
\bibliographystyle{plain}

\end{document}